\documentclass[VANCOUVER,STIX1COL]{WileyNJD-v2}

\usepackage{amsfonts}
\usepackage{amsmath}
\usepackage{subfig}
\usepackage{graphicx}
\usepackage{epstopdf}

\articletype{Research Article}%

\received{18 April 2023}
\revised{2 December 2023}
\accepted{}

\begin{document}

\title{Anderson acceleration with approximate calculations: applications to scientific computing\protect\thanks{This manuscript has been authored in part by UT-Battelle, LLC, under contract DE-AC05-00OR22725 with the US Department of Energy (DOE). The US government retains and the publisher, by accepting the article for publication, acknowledges that the US government retains a nonexclusive, paid-up, irrevocable, worldwide license to publish or reproduce the published form of this manuscript, or allow others to do so, for US government purposes. DOE will provide public access to these results of federally sponsored research in accordance with the DOE Public Access Plan (\url{http://energy.gov/downloads/doe-public-access-plan}).}}

\author[1]{Massimiliano Lupo Pasini*}

\author[2]{M. Paul Laiu}

\authormark{Massimiliano Lupo Pasini \textsc{et al}}

\address[1]{\orgdiv{Computational Sciences and Engineering Division}, \orgname{Oak Ridge National Laboratory}, \orgaddress{1 Bethel Valley Road, \state{Tennessee}, 37831, \country{USA}}}

\address[2]{\orgdiv{Computer Science and Mathematics Division}, \orgname{Oak Ridge National Laboratory}, \orgaddress{1 Bethel Valley Road, \state{Tennessee}, 37831, \country{USA}}}

\corres{*Massimiliano Lupo Pasini, 1 Bethel Valley Road, Mail Stop 6085, P.O. Box 2008, Oak Ridge, TN, 37831, USA. \email{lupopasinim@ornl.gov}}

\presentaddress{1 Bethel Valley Road, Mail Stop 6085, P.O. Box 2008, Oak Ridge, TN, 37831, USA}

\abstract[Summary]{We provide rigorous theoretical bounds for Anderson acceleration (AA) that allow {for approximate calculations when applied to solve linear problems.
We show that, when the approximate calculations satisfy the provided error bounds, the convergence of AA is maintained while the computational time could be reduced.}
{We also provide computable heuristic quantities, guided by the theoretical error bounds, which can be used to automate the tuning of accuracy while performing approximate calculations. For linear problems}, the use of heuristics to monitor the error introduced by approximate calculations, combined with the check on monotonicity of the {residual}, ensures the convergence of the numerical scheme within a prescribed {residual} tolerance. 
%{Whenever the underlying fixed-point scheme converges, our procedure ensures that AA with approximate calculations maintains converges also for non-linear problems.}
{Motivated by the theoretical studies}, we propose a reduced variant of AA, which consists in projecting the least-squares used to compute the Anderson mixing onto a subspace of reduced dimension. 
The dimensionality of this subspace adapts dynamically at each iteration as prescribed by the computable heuristic quantities. 
We numerically show and assess the performance of AA with approximate calculations on: (i) linear deterministic fixed-point iterations arising from the Richardson's scheme to solve linear systems with open-source benchmark matrices with various preconditioners and (ii) non-linear deterministic fixed-point iterations arising from non-linear time-dependent Boltzmann equations.}


\maketitle

\footnotetext{\textbf{Abbreviations:} AA, Anderson Acceleration; AAR, Alternating Anderson-Richardson; AR, Anderson-Richardson; FOM, Full Orthogonalization Method; GMRES, Generalized Minimum Residual; HPC, high-performance computing; ILU, Incomplete LU factorization; ILUT, Incomplete LU factorization with thresholding; IMEX, implicit-explicit; LHS, left-hand side; MM, Matrix Market; RHS, right-hand side}

\section{Introduction}
Efficient fixed-point schemes are needed in many complex physical applications such as: (i) iterative algorithms to solve sparse linear systems, (ii) Picard iterations \cite{Kewang2022} to solve systems of non-linear partial differential equations occurring in computational fluid dynamics \cite{rebholz, 10.1093/imanum/draa095}, semiconductor modelling \cite{doi:10.1137/20M1365922, LAIU2020109567}, and astrophysics \cite{Paul_Laiu_2020, Paul_Laiu_2021}, and (iii) solving extended-space full-waveform inversion (FWI) \cite{dx.doi.org/10.1190/geo2021-0409.1}.
%, and (iv) iterative training of physics informed deep learning (DL) models using stochastic optimization techniques \cite{scieur_online_2019, Nonlinear_fixed_point, geist_anderson_2018, 8461063, sun2021damped, AADL_paper}.

For scientific applications, the standard fixed-point scheme often converges slowly because the fixed-point operator is either weakly globally contractive or is not globally contractive at all. 
The Anderson acceleration (AA) \cite{Anderson} is a multi-secant method \cite{brezinski_shanks_2018, 10.1093/imanum/drab061} that has been widely used either to improve the convergence rate of convergent fixed-point schemes or to restore convergence when the original fixed-point scheme is not convergent \cite{Fang, walker_anderson_2011, potra_characterization_2013, zhang2020, saad2021}. In particular, the convergence of AA has been studied with respect to specific properties of different scientific applications \cite{toth_convergence_2015, toth_local_2017, Mai2019, doi:10.1137/20M132938X}.

AA requires solving a least-squares {problem} at each fixed-point iteration, and this can be computationally expensive for scientific applications that involve large-scale calculations, especially when such calculations are distributed on high-performance computing (HPC) platforms.
To reduce the computational cost of AA, a recently proposed variant of AA called Alternating Anderson Richardson (AAR) \cite{banerjee_periodic_2016, AAJ, AAR} performs multiple fixed-point iterations between two consecutive Anderson mixing corrections. The number of fixed-point iterations between two consecutive Anderson mixing corrections can be arbitrarily set by the user. When the number is set to one, AAR reduces back to AA. Recent theoretical results have shown that AAR is more robust than the standard AA against stagnations \cite{lupo_pasini_convergence_2019}. 
However, even if performed intermittently, computing the Anderson mixing via least-squares can still be computationally expensive when the number of rows in the tall and skinny matrix that defines the left-hand side (LHS) of the least-squares problem is large. 
One possibility to reduce the computational cost is to perform approximate/inexact calculations in the AAR scheme. To avoid compromising the convergence of AAR, the computation accuracy reduction must be performed judiciously. 
For specific problems, the accuracy has been reduced by projecting the least-squares {problem} in the Anderson mixing computation onto an appropriate projection subspace, which results in computational savings without affecting the final convergence of the fixed-point scheme \cite{doi:10.1137/19M1290097, doi.org/10.1111/cgf.14081}. 
{In these works, the problem is formulated as a non-linear non-convex constrained optimization problem, which is recast into a fixed-point problem with a computationally inexpensive fixed-point operator based on the alternating direction method of multipliers (ADMM).} The accuracy reduction adopted to reduce the computational time to evaluate the original fixed-point operator has been controlled by using physics information about the problem at hand to identify an effective projection subspace. However, physics-driven guidelines are difficult (and often impossible) to determine because of the lack of structure in the fixed-point operator. In this situation, general guidelines (not necessarily physics driven) to judiciously reduce the accuracy in solving the least-squares problem are needed. 
%{For example, a broader context than the one describe in \cite{doi:10.1137/19M1290097, doi.org/10.1111/cgf.14081} occurs when a large scale linear system is solved using an iterative numerical scheme on distributed computing HPC resources. This is a necessary computational step whenever a non-linear fixed point scheme is used to iteratively update variables of a set of non-linear partial differential equations (PDEs). In fact, large scale linear systems must be solved to update velocity and pressure in Navier-Stokes equations for computational fluid dynamics, electronic distribution and exchange correlation functional in Kohn-Sham equations for density functional theory, and neutrino distribution in Boltzmann equations for computational astrophysics. By recasting the linear system as a linear fixed-point, the overall numerical approach results into two nested fixed-point schemes. The most external non-linear fixed point scheme updates the variables that characterize the non-linear PDE, whereas the internal linear fixed point scheme solves the linear systems to compute the updates for each variable. The internal linear fixed-point operator associated with each linear system is computationally inexpensive to evaluate (even on large scale distributed compute resources) because its evaluation requires only one matrix-vector multiplication. Applying AAR to accelerate the internal linear fixed point scheme results into a numerical scheme where the computational time to solve the least-squares dominates over the computational time to evaluate the linear fixed-point operator.}

We provide rigorous theoretical bounds for AAR on \textit{linear} fixed-point iterations that establish general guidelines to reduce the accuracy of the calculations performed by AAR while still ensuring that the final residual of the fixed-point scheme drops below a user-defined convergence threshold. 
By interpreting the accuracy reduction in AAR calculations as a perturbation to the original least-squares {problem} in the Anderson mixing computation, we assess how much accuracy can be sacrificed to limit the communication and computational burden without compromising convergence. Along the same lines as
previously published theoretical results in \cite{simoncini} for well-established Krylov methods \cite{Saad}, our analysis concludes that the convergence of AAR can be maintained while allowing to progressively reduce the accuracy for the calculations within the AAR scheme. 

Our theoretical results allow for accuracy reduction in different calculations performed by AAR {on linear fixed-point problems}. When the fixed-point operator evaluations are the dominant computational cost of AAR, one may choose to approximate the evaluations of the fixed-point operator to reduce the computational cost of the evaluations. When instead solving the least-squares problem for the Anderson mixing is the most expensive step of AAR, one may solve the least-squares problem approximately to reduce the computational cost of the least-squares solver. 
Our theoretical results open a new path to efficiently apply AAR to solve problems in addition to sparse linear systems, including non-linear fixed-point iterations arising from quantum mechanics, plasma astrophysics, and computational fluid dynamics where fixed-point operator evaluations are often expensive \cite{toth_local_2017}. In this context, the use of AAR allows to leverage less expensive operator evaluations without affecting the final attainable accuracy of the non-linear physics solver. 

Since the theoretical bounds are defined in terms of quantities that are not always easily computable, we also propose computationally inexpensive heuristics while still maintaining a close relation to the rigorous theoretical estimates. The heuristics allow to dynamically adjust the accuracy in the AAR calculations {with minimal} computational overhead. 
We combine the heuristics with a check on the monotonic reduction of the residual norm across two consecutive Anderson mixing steps to ensure that the accuracy reduction in the AAR calculations does not compromise the convergence of the numerical scheme. 
The monotonicity check allows one to replace the tight theoretical bound with an empirical one, and dynamically tighten the tolerance of inexact calculations if non-monotonicity of the residual norm is observed across two consecutive AA steps. 

In the numerical sections, we assess the effectiveness of our heuristics using two different approaches to inject inaccuracies into AAR. The first approach injects error in the evaluations of the fixed-point operator, and the heuristics are used to dynamically adjust the magnitude of the injected error. The second approach projects the least-squares problem in the Anderson mixing computation onto a subspace, and the heuristics are used to dynamically adjust the dimension of projection subspace. 
Numerical results confirm that using approximate calculations to solve the least-squares {problem} for AA saves computational time without sacrificing convergence to a desired accuracy. 
The proposed method has appealing properties for HPC since it reduces computational requirements, inter-process communications, and storage requirements to solve least-squares {problem} in AA when the fixed-point operator is distributed across multi-node processors. 

The remainder of the paper is organized as follows. In Section~\ref{section:background} we briefly recall the stationary Richardson's method that solves linear systems by recasting them as linear fixed-point iterations. 
In Section~\ref{section:error_bounds} we provide rigorous error bounds that allow approximate calculations of AA for linear fixed-point iterations while still maintaining convergence.
In Section~\ref{section:reduced_aar} we {propose a new variant of AAR, which performs inexpensive approximate calculation by projecting the least-squares {problem} onto a subspace}, called Reduced AAR. 
{
In Section~\ref{section:numerical_results}, we illustrate the convergence property of AAR on linear fixed-point problems and demonstrate the computational advantage of the Reduced AAR solvers on examples including both linear and non-linear fixed-point problems. Even though the theoretical results are not directly applicable to the non-linear fixed-point solvers, the promising numerical results reported in Section~\ref{section:numerical_results} encourage further investigations.} 
%by numerical implementations of the analyses conducted in the previous sections for two different scientific applications that we use as representatives of two different types of fixed-point iterations. The applications we consider are: (i) linear deterministic fixed-point iterations arising from the Richardson's scheme to solve linear systems and (ii) non-linear deterministic fixed-point iterations arising from non-linear time-dependent Boltzmann equations.
%, and (iii) non-linear stochastic fixed-point iterations arising from the training of neural networks model using batched stochastic optimization.
We conclude the paper with remarks on the state-of-the-art and comments about future developments in Section~\ref{section:conclusions}.

\section{Fixed-point iteration}
The standard iterative method for solving a fixed-point problem
$
\mathbf{x} = G(\mathbf{x})$ with $\mathbf{x}\in\mathbb{R}^n$ and $G:\mathbb{R}^n \rightarrow \mathbb{R}^n$ is the fixed-point iteration:
\begin{equation}
\mathbf{x}^{k+1} = G(\mathbf{x}^{k}), \quad k=0,1,\ldots,
\label{fixed_point}
\end{equation}
and the residual $\mathbf{r}^k \in \mathbb{R}^n$ of the fixed-point iteration is defined as 
\begin{equation}
\mathbf{r}^k = G(\mathbf{x}^k)-\mathbf{x}^{k} , \quad k=0,1,\ldots\:.
\end{equation}
The convergence of the fixed-point iteration relies on the global contraction property of the non-linear fixed-point operator $G$, { i.e., there is a constant $c<1$ such that, for every $\mathbf{x}\in \mathbb{R}^n$ and $\mathbf{y}\in \mathbb{R}^n$, $\lVert G(\mathbf{x}) - G(\mathbf{y})\rVert_2\le c \lVert \mathbf{x} - \mathbf{y}\rVert_2$. When such $c$ is close to one,} the fixed-point iteration may converge at an unacceptably slow rate.

\subsection{Anderson acceleration and the Alternating Anderson acceleration}
The AA was proposed to accelerate the fixed-point iteration \cite{Anderson}. There exist several equivalent formulations of AA \cite{Fang, toth_convergence_2015}. The formulation we adopt for convenience in the analysis is described in Algorithm \ref{AAR_algorithm}.
Define
\begin{equation}
    X_k = [(\mathbf{x}^{k-\ell+1} - \mathbf{x}^{k-\ell}), \ldots, (\mathbf{x}^k-\mathbf{x}^{k-1}) ]\in \mathbb{R}^{n\times \ell}
\end{equation}
and 
\begin{equation}
    R_k = [(\mathbf{r}^{k-\ell+1} - \mathbf{r}^{k-\ell}), \ldots, (\mathbf{r}^k-\mathbf{r}^{k-1}) ]\in \mathbb{R}^{n\times \ell},
\end{equation}
where $\ell=\min\{k,m\}$ and $m$ is an integer that describes the maximum number of previous terms of the sequence $\{\mathbf{x}^k\}_{k=0}^{\infty}$ used to compute the AA update. Denoting with $\lVert \cdot \rVert_2$ the $\ell^2$-norm of a vector, the vector $\mathbf{g}^k = [g^{(k)}_1, \ldots, g^{(k)}_{\ell}]^T \in \mathbb{R}^\ell$ defined as
\begin{equation}
    \mathbf{g}^k=\underset{\mathbf{g}\in \mathbb{R}^{\ell}}{\operatorname{argmin}}\lVert R_k\mathbf{g} - \mathbf{r}^k\rVert_2^2
   \label{least_squares_matrix_form}
\end{equation}
is used to update the sequence $\{\mathbf{x}^k\}_{k=0}^{\infty}$ through an Anderson mixing as follows
\begin{equation}
    \mathbf{x}^{k+1} = \mathbf{x}^k + {\mathbf{r}^k} - (X_k+R_k) \mathbf{g}^k.
    \label{anderson_original}
\end{equation}

\begin{algorithm}
\caption{(Alternating) Anderson Acceleration}
\label{AAR_algorithm}
\begin{algorithmic}
\State{Data: $\mathbf{x}^0\in\mathbb{R}^n$, $p\in\mathbb{N}$; \hfill \Comment{$// \texttt{AA: } p=1,\,\,\texttt{Alternating AA: } p>1$}}
\State{Compute $\mathbf{r}^0 = G(\mathbf{x}^0)-\mathbf{x}^0$ and $\mathbf{x}^1=\mathbf{x}^0+\mathbf{r}^0$;}
\State{Set $k=1$}
\While{$\displaystyle \lVert \mathbf{r}^{k-1} \rVert_2>tol$}
 \State{Compute $\ell=\min\{k,m\}$}
 \State{Compute $\mathbf{r}^k=G(\mathbf{x}^k)-\mathbf{x}^k$}
 \If{$k\pmod p\ne 0$}
 \State{Update $\mathbf{x}^{k+1}=\mathbf{x}^k + \omega\mathbf{r}^k$}
 \Else
  \State{Set $X_k = [(\mathbf{x}^{k-\ell+1} - \mathbf{x}^{k-\ell}), \ldots, (\mathbf{x}^k-\mathbf{x}^{k-1}) ]\in \mathbb{R}^{n\times \ell}$}
 \State{Set $R_k = [(\mathbf{r}^{k-\ell+1} - \mathbf{r}^{k-\ell}), \ldots, (\mathbf{r}^k-\mathbf{r}^{k-1}) ]\in \mathbb{R}^{n\times \ell}$}
 \State{Compute  
 $\displaystyle
 \mathbf{g}^k=\underset{\mathbf{g}\in \mathbb{R}^{\ell}}{\operatorname{argmin}}\lVert R_k\mathbf{g} - \mathbf{r}^k\rVert_2^2
 $}
 \State{Update 
  ${\mathbf{x}^{k+1} = \mathbf{x}^k + \mathbf{r}^k - (X_k+R_k) \mathbf{g}^k}$.}
 \EndIf
 \State{$k=k+1$}
\EndWhile
 \Return{$\mathbf{x}^{k+1}$ }
\end{algorithmic}
\end{algorithm}

When compared to the fixed-point iteration, AA often requires fewer iterations to converge thus resulting in a shorter computational time. On the other hand, AA also introduces the overhead of solving the least-squares problem \eqref{least_squares_matrix_form} at each iteration.
This computation overhead is outweighed by the benefit from fewer iterations when solving problems in which evaluating the operator $G$ incurs the dominant cost. However, there are many problems where the cost of solving this least-squares problem incurs the main cost.

Solving a least-squares problem at each iteration is computationally expensive and requires global communications which introduce severe bottlenecks for the parallelization in HPC environments. 
P.~Suryanarayana and collaborators {\cite{AAR}} recently proposed to compute
an Anderson mixing after multiple Picard iterations, so that the cost of solving successive least-squares problems is reduced. This new AA variant, called \textit{Alternating AA}, has been shown to effectively accelerate both linear and non-linear fixed-point iterations \cite{AAJ, AAR}. Algorithm \ref{AAR_algorithm} describes the Alternating AA, where the parameter $p$ represents the number of Picard iterations separating two consecutive AA steps. Letting $p\rightarrow \infty$ reduces the Alternating AA to a standard fixed-point scheme with $\omega$ the relaxation parameter, while $p=1$ makes it coincide with the standard AA formulation.

In sections \ref{section:background} and \ref{AA_linear} we focus on the case when $G$ is linear and arises from iteratively solving a sparse linear system. 

\subsection{Stationary Richardson}
\label{section:background}
Consider a nonsingular sparse linear system 
\begin{equation}
A\mathbf{x}=\mathbf{b},
\label{linsys}
\end{equation}
where $A\in\mathbb{R}^{n\times n}$ and $\mathbf{x}$, $\mathbf{b} \in \mathbb{R}^n$. 
We assume that left preconditioning has already been applied.
Let $H := I-A$; adding $H\mathbf{x}$ on both sides of \eqref{linsys} leads to a linear fixed-point problem with the fixed-point iteration
\begin{equation}
\mathbf{x}^{k+1} = G(\mathbf{x}^k) := H\mathbf{x}^k + \mathbf{b},\quad k=0,1,\ldots\:.
\label{linear_fixed_iteration}
\end{equation}
A commonly used equivalent representation of \eqref{linear_fixed_iteration} is the correction form
\begin{equation}
\mathbf{x}^{k+1} =\mathbf{x}^k + \mathbf{r}^k,\quad k=0,1,\ldots,
\label{fixedpoint}
\end{equation}
where $\mathbf{r}^k = \mathbf{b}-A\mathbf{x}^k$ is the residual at the $k$th iteration.
The scheme \eqref{fixedpoint} converges to the solution if and only if the spectral radius $\rho(H)<1$. 
The scheme \eqref{linear_fixed_iteration} can be generalized by introducing a positive weighing parameter $\omega$:  
\begin{equation}
\mathbf{x}^{k+1}=(1-\omega)\mathbf{x}^k + \omega( H\mathbf{\mathbf{x}}^k+\mathbf{b}),
\label{weighted}
\end{equation}
which in correction form is equivalently represented as
\begin{equation}
\mathbf{x}^{k+1} = \mathbf{x}^k + \omega \mathbf{r}^{k}.
\label{weighted2}
\end{equation}
Equations \eqref{weighted} and \eqref{weighted2} are known as the \textit{stationary Richardson scheme}. 
{ It is also straightforward to verify that Equations~\eqref{weighted} and \eqref{weighted2} are equivalent to applying the fixed-point iteration to the linear operator $G_\omega(\mathbf{x}):=(I-\omega A)\mathbf{x} + \omega\mathbf{b}$.}

\subsection{Anderson acceleration for linear systems}
\label{AA_linear}
The AA is well suited to potentially accelerate Richardson schemes. 
In the case of a linear fixed-point iteration as in Equation \eqref{linear_fixed_iteration}, AA can be split into two steps. The first step consists in calculating the Anderson mixing
\begin{equation}
\overline{\mathbf{x}}^k = \mathbf{x}^k - X_k\mathbf{g}^k,
\label{anderson_correction}
\end{equation}
where the vector $\mathbf{g}^k$ is computed by solving the over-determined least-squares {problem} in Equation \eqref{least_squares_matrix_form}.
Under the assumption that $R_k$ is full rank, we have
$ \mathbf{g}^k  = (R^T_kR_k)^{-1}R^T_k\mathbf{r}^{k}$.
Since $R_k = -AX_k$ for a linear fixed-point iteration, \eqref{anderson_correction} can be recast as 
\begin{equation}
\overline{\mathbf{x}}^k 
= \mathbf{x}^k - X_k(R^T_kR_k)^{-1}R^T_k\mathbf{r}^{k} 
= \mathbf{x}^k + X_k(R^T_k A X_k)^{-1}R^T_k\mathbf{r}^{k}.
\end{equation}
Therefore, the Anderson mixing can be interpreted as a step of an \textit{oblique projection method} \cite[Chapter~5]{Saad}
with the subspace of corrections chosen as
$
\mathcal{V}_k = \mathcal{R}(X_k)
$
and the subspace of constraints as
$
\mathcal{W}_k = \mathcal{R}(R_k) = A \mathcal{R}(X_k) = A \mathcal{V}_k
$.
After the Anderson mixing is applied, $\overline{\mathbf{x}}^k$ is used to compute a new update via a standard Richardson's step
\begin{equation}
 \mathbf{x}^{k+1} = \overline{\mathbf{x}}^k + \omega \overline{\mathbf{r}}^k,
 \label{relaxation_step}
\end{equation}
where $\overline{\mathbf{r}}^k = \mathbf{b}-A \overline{\mathbf{x}}^k$.
It can be shown that the iterative scheme from Equations~\eqref{anderson_correction} and \eqref{relaxation_step} is equivalent to applying the standard Anderson mixing \eqref{anderson_original} to a fixed-point problem with operator $G_\omega(\mathbf{x}):=(I-\omega A)\mathbf{x} + \omega\mathbf{b}$.
If AA is applied at each fixed-point iteration ($p=1$), 
this scheme is called \textit{Anderson-Richardson} (AR) \cite{Anderson, Potra, Walker}.
Studies about AR have highlighted similarities between this method and GMRES \cite{Fang, fang_two_2009, Potra, walker_anderson_2011}. 
When $p>1$, we refer to this approach as AAR. Recent studies have shown also that AAR is more robust against stagnations than AR%
\footnote{
As in the earlier work \cite{Potra}, \textit{stagnation} is defined here as a situation where the approximate solution $\mathbf{x}^k$ does not change across consecutive iterations. The number of consecutive iterations in which the approximate solution does not change represents the extension of the stagnation.} \cite{lupo_pasini_convergence_2019}.

\subsubsection{Computational cost of AAR}
\label{cost_aar}
Denoting the {average} number of non-zero entries per row of a sparse $n\times n$ matrix {$A$ as $\texttt{nnz}(A)$} and assuming for simplicity that this number is roughly constant across the rows, the computational complexity of a sparse matrix-vector multiplication is $\mathcal{O}({\texttt{nnz}(A)\,}n)$. 
The least-squares problem in Equation \eqref{least_squares_matrix_form} used to perform AA is {solved} using the QR factorization \cite{golub13} of the $n\times m$ matrix $R_k$. 
The main components used in Algorithm \ref{AAR_algorithm} to perform AA with their computational costs expressed with the big-$\mathcal{O}$ are:
\begin{itemize}
    \item QR factorization to perform AA: $\mathcal{O}(nm^2)$
    \item matrix-vector multiplication: $\mathcal{O}({\texttt{nnz}(A)\,}n)$.
\end{itemize}
AAR mitigates the computational cost to perform an AA step by interleaving successive AA step with $p$ relaxation sweeps. 
Relaxing the frequency at which the least-squares {problem} for AA is {solved} reduces the computational cost per iteration, but it is recommended not to excessively relax the frequency to avoid severely deteriorating the convergence rate of the overall numerical scheme. { In practice, $p$ is often set to a value between three and six \cite{lupo_pasini_convergence_2019}}.
The {average} computational complexity of AAR is
\begin{equation}
    \mathcal{O}\bigg( {\texttt{nnz}(A)\,} n + \frac{1}{p} nm^2\bigg).
\end{equation}
We can see the $\frac{1}{p}$ factor in front of the computational cost of the QR factorization allows to mitigate the impact that the least-squares solve has over the total computational cost of the numerical scheme.

In large-scale distributed computational environments, solving the global least-squares {problem} at periodic intervals may still not be sufficient to avoid severe bottlenecks in the computation. 
% Further reducing the computational cost to solve the least-squares {problem} is possible by replacing expensive and accurate calculations with inexpensive approximate ones while ensuring that convergence is maintained. 
{To further reduce the computational cost, one could choose to adopt less expensive, approximate fixed-point operator evaluations or to solve the least-squares problems approximately.
In both cases, these approximate calculations could often be formulated as injections of errors in the residual terms $R_k$ and $\mathbf{r}^k$ in the least-squares problems.}
To this end, in the following section we conduct an error analysis that provides theoretical bounds to estimate the error affordable at each iteration of AAR, which will allow us to develop numerical strategies to reduce the cost of the least-squares solves.

\section{Backward stability analysis of approximate least-squares solves for AAR}
\label{section:error_bounds}
The solution obtained by solving approximately the least-squares problem \eqref{least_squares_matrix_form} can be treated as the exact solution of a different least-squares problem, which can be looked at as a yet unknown perturbation of the original problem \eqref{least_squares_matrix_form}. 
The objective of this section is to estimate the difference between the original least-squares problem \eqref{least_squares_matrix_form} and its perturbation. This is important in order to ensure backward stability \cite{saunders, doi:10.1137/17M1157106}, which is a desired property of the numerical scheme. 
{By combining the backward stability and the conditioning of the problem, we then provide a forward error estimate at the end of this section (see Corollary~\ref{cor:forward}).}
{As discussed in the last section, approximate calculations considered here} may arise from approximate evaluations of the fixed-point operator or from approximate calculations in the solver for the least-squares {problem} \eqref{least_squares_matrix_form}. 

Analogous to the approach adopted in the convergence analysis of AAR \cite{lupo_pasini_convergence_2019}, we provide the backward stability analysis under the assumption that the AAR is not truncated, i.e., $m=k$, which henceforth we refer to as ``Full AAR''. {We note that the $m=k$ assumption is needed to show that, on linear problems, AA is equivalent to GMRES barring stagnations. \cite{walker_anderson_2011} Hence, it is common to assume $m=k$ in the convergence proofs for variants of AA on linear problems.\cite{lupo_pasini_convergence_2019}}%
\footnote{
%In \cite{lupo_pasini_convergence_2019}, the assumption that the AAR is not truncated was demonstrated to be necessary to ensure final convergence of AAR on a linear systems with a general non-singular matrix $A$. Since the goal of our backward analysis is to guide the performance of approximate calculations without affecting the final convergence, we preserve this assumption also for the following backward analysis.
%If the additional assumption of $A$ being positive definite holds, then the Augmented AAR (AAAR) is guaranteed to converge with a strictly monotonic decrease of the Euclidean norm of the residual in any of its truncated variants, including when $m=1$. 
The $m=k$ assumption is made so that we can leverage the existing convergence analysis in \cite{lupo_pasini_convergence_2019} for AAR when solving linear systems, which extends equivalence between standard AA and GMRES shown in \cite{walker_anderson_2011}, also under this $m=k$ assumption.
For linear problems with positive definite matrix $A$, the $m=k$ assumption is not needed in the convergence proof of the Augmented AAR solver \cite[Section 3.3]{lupo_pasini_convergence_2019}.
When $A$ is positive definite, the backward stability analysis considered in this section extends to the truncated version of the Augmented AAR solver.} 
We focus on the backward stability analysis of Full AAR in the case that the mixing coefficients $\mathbf{g}^k$ solves an inexact or perturbed version of \eqref{least_squares_matrix_form}.
Specifically, we consider the following two cases:
(i) {when backward perturbations are permitted only to the} matrix $R_k$ on the left-hand side (LHS) of the least-squares system, which covers the scenario when the numerical methods performs an approximate factorization of $R_k$ (e.g., approximate QR or low-rank singular value decomposition), and
(ii) {when backward perturbations are permitted to} both the LHS matrix $R_k$ and the residual vector $\mathbf{r}^k$ on the right-hand side (RHS), {which occurs when matrix-vector multiplications are performed inexactly and/or} the least-squares {problem} \eqref{least_squares_matrix_form} is approximated by projection onto a subspace. 
The backward stability analyses for cases (i) and (ii) are given in Sections~\ref{inaccuracy_lhs} and \ref{inaccuracy_both}, respectively.%
\footnote{
In scientific applications where $m$ is fixed at $m<k$ and the history window of updates is cyclically erased, the theoretical results provided in this work can still be applied by restarting the iteration count from scratch every time the memory is erased.     
}

Let us denote with $\mathcal{E}_k \in \mathbb{R}^{n\times k}$ any perturbation to the matrix $R_k$ and $\delta \mathbf{r}_k \in\mathbb{R}^n$ any perturbation to the residual $\mathbf{r}^k$.
The $i$th column of $\mathcal{E}_k$ can be modeled as an error matrix $E_i\in\mathbb{R}^{n\times n}$ applied to the $i$th columns of $R_k$, i.e.,
\begin{equation}
    \mathcal{E}_k = [E_1{ (\mathbf{r}^1-\mathbf{r}^0)}, \ldots, E_k{  (\mathbf{r}^{k}-\mathbf{r}^{k-1})}]\:.
\end{equation}    
We then denote the perturbed LHS as
\begin{equation}
    \hat{R}_k = R_k + \mathcal{E}_k = R_k + [E_1{ (\mathbf{r}^1-\mathbf{r}^0)}, \ldots, E_k{  (\mathbf{r}^{k}-\mathbf{r}^{k-1})}].
\end{equation}
Using $\mathbf{r}^i - \mathbf{r}^{i-1} = -A(\mathbf{x}^i - \mathbf{x}^{i-1})$ for $i=1,\ldots,k$, 
we obtain
\begin{equation}
    \hat{R}_k = R_k {  -} [E_1{ A}(\mathbf{x}^1-\mathbf{x}^0), \ldots, E_k{  A }(\mathbf{x}^{k}-\mathbf{x}^{k-1})].
    \label{eq:perturbed_LHS}
\end{equation}
As for the perturbed RHS, we denote it as 
\begin{equation}
    \hat{\mathbf{r}}^k = \mathbf{r}^k + \delta \mathbf{r}^k\:.
    \label{eq:perturbed_RHS}
\end{equation}

{

To carry out the backward stability analysis in the following sections, we make the following assumption which will remain valid for the remainder of this paper.
\begin{assumption}
\label{FullrankAssumption}
The matrix $R_k$ is full rank, and the perturbation $\mathcal{E}_k$ does not compromise the full-rank property of $\hat{R}_k$.%
\end{assumption}
The full-rank property of $R_k$ is not needed to ensure convergence of Full AAR \cite{lupo_pasini_convergence_2019}. However, we need this property in the following theoretical analysis in order to ensure backward stability. In fact, requiring that $R_k$ is full-rank already excludes the risk of stagnation \cite{lupo_pasini_convergence_2019}. The same reasoning applies when $R_k$ is replaced by $\hat{R}_k$ to allow for inaccuracy in the least-squares calculations. 
}

We denote the QR factorizations \cite{Golub} of $R_k$ and $\hat{R}^k$ computed with a backward stable numerical method such as modified Gram-Schmidt or Householder transformations respectively as
\begin{equation}
    R_k = Q_k T_k \quad\text{and}\quad \hat{R}_k = \hat{Q}_k \hat{T}_k,
\end{equation}
where $Q_k \in \mathbb{R}^{n\times k}$ and $\hat{Q}_k \in \mathbb{R}^{n\times k}$ are matrices with orthogonal columns and $T_k \in \mathbb{R}^{k\times k}$
and $\hat{T}_k \in \mathbb{R}^{k\times k}$ are upper triangular matrices.

\subsection{Backward stability analysis of AAR with perturbed LHS}
\label{inaccuracy_lhs}
In this subsection, we consider the case where the vector of mixing coefficients {$\mathbf{g}^k$ in \eqref{anderson_correction} is approximated by the solution ${\hat{\mathbf{g}}^k}$ of a perturbed least-squares problem}
\begin{equation}
  {\hat{\mathbf{g}}^k} =\underset{\mathbf{g}\in \mathbb{R}^{k}}{\operatorname{argmin}}\lVert  \hat{R}_k  \mathbf{g} - \mathbf{r}^k  \rVert_2^2,
  \label{least_squares_perturbed_matrix}
\end{equation}
where the perturbed LHS matrix $\hat{R}_k$ is defined in \eqref{eq:perturbed_LHS}.
The optimal backward error of the least-squares in \eqref{least_squares_perturbed_matrix} is defined as \cite{saunders} 
\begin{equation}
    \min_{\mathcal{E}_k}\{ \lVert \mathcal{E}_k \rVert_F \; \text{:} \; \lVert  (R_k + \mathcal{E}_k)   {\hat{\mathbf{g}}^k} - \mathbf{r}^k  \rVert_2^2 = \min\},
\end{equation}
where $\lVert \cdot \rVert_F$ denotes the Frobenius norm. 
{Denote with $z$ a positive integer}. When the Anderson mixing is performed at iteration {$k=z p$}, 
we define the perturbation from the backward error as 
\begin{equation}
    \delta_k = \| \mathcal{E}_k {\hat{\mathbf{g}}^k} \|_2 = 
    \lVert [E_1{ A}(\mathbf{x}^1-\mathbf{x}^0), \ldots, E_k{ A}(\mathbf{x}^{k}-\mathbf{x}^{k-1})]{ {\hat{\mathbf{g}}^k} } \rVert_2.
    \label{distance}
\end{equation}
Assume that $k={z} p$ iterations of the AAR have been carried out with approximate residual evaluations. 
The perturbation $\delta_k$ is bounded by 
\begin{equation}
    \delta_k 
    \le \sum_{i=1}^k \lvert {\hat{g}_i}^{(k)}\rvert \lVert E_i \rVert_2 { \lVert A \rVert_2} \lVert \mathbf{x}^i-\mathbf{x}^{i-1}\rVert_2,
\end{equation}
where $\lVert \cdot \rVert_2$ denotes the matrix norm induced by the $\ell^2$-norm of a vector. 
% Given $\epsilon \in (0,1)$, we consider the perturbation matrices $E_1, \ldots, E_k$ that satisfy the following chain of inequalities:
% \begin{equation}
% \lVert [E_1{ A}(\mathbf{x}^1-\mathbf{x}^0), \ldots, E_k{ A}(\mathbf{x}^{k}-\mathbf{x}^{k-1})] \rVert_2 \le \epsilon { \lVert A \rVert_2} \lVert X_k \rVert_2 \le \sum_{i=1}^k\epsilon { \lVert A \rVert_2} \lVert \mathbf{x}^i-\mathbf{x}^{i-1} \rVert_2.   
% \label{reviewer1_comment}
% \end{equation}
% To simplify the presentation, we shall incorporate the norm of $A$ in the tolerance, thus defining $\varepsilon = \epsilon \lVert A \rVert_2$, $\varepsilon \in (0,\lVert A \rVert_2)$. 

% {{
% \begin{remark}
% In scientific computing applications, the degree of inaccuracy injected in \eqref{reviewer1_comment} due to approximate calculations is always an arbitrarily tunable parameter. The only factor that limits the tuning of the magnitude of the inaccuracy from below is the finite machine precision, which is a factor beyond the control of the user. However, when a calculation is defined as approximate in the context of scientific computing, the implicit underlying assumption is that the amount of error injected in \eqref{reviewer1_comment} is larger than the finite machine precision. 
% \end{remark}
% }
% }

The following theoretical results provide error bounds on the affordable perturbation introduced into the least-squares problem by approximate calculations to ensure that the residuals of AAR iterates on the linear system \eqref{linsys} decreases below a user-defined threshold $\epsilon$.
\begin{lemma}
\label{lemma1}
Assume that $k={z}p$ iterations of Full AAR have been carried out. Let ${\hat{\mathbf{g}}}^{k}$ be the Anderson mixing computed by solving the perturbed least-squares problem defined in Equation \eqref{least_squares_perturbed_matrix}. Then, the following inequality holds
\begin{equation}    \label{estimate_lemma1}
    \lvert {\hat{g}}_i^{(k)} \rvert \le  \frac{1}{\sigma_{\min}(\hat{T}_k)} \lVert \mathbf{r}^k \rVert_2 \quad \text{for any}\quad i= 1,\ldots,k. 
\end{equation}
\begin{proof}
{ Let us denote with $\mathbf{r}^k_{\text{proj}}$ the $\ell^2$-orthogonal projection of the residual $\mathbf{r}^k$ onto the column space of $\hat{R}_k$.} The use of the projected residual allows to recast the perturbed least-squares problem in Equation \eqref{least_squares_perturbed_matrix} as 
$\hat{R}_k {\hat{\mathbf{g}}^k} = {\mathbf{r}^k_{\text{proj}}}$. 
By replacing the matrix $\hat{R}_k$ with its factorization $\hat{R}_k=\hat{Q}_k \hat{T}_k$ in the least-squares problem 
and applying $(\hat{T}_k)^{-1} \hat{Q}_k^T$ on both sides of ${ \hat{R}_k \hat{\mathbf{g}}^k} = {\mathbf{r}^k_{\text{proj}}}$,
we obtain 
\begin{equation}
   {\hat{\mathbf{g}}^k}= (\hat{T}_k)^{-1} \hat{Q}_k^T {\mathbf{r}^k_{\text{proj}}}.
\end{equation}
Then it follows that
\begin{equation}
\begin{aligned}
\lVert {\hat{\mathbf{g}}^k}\rVert_2 & = \lVert (\hat{T}_k)^{-1} \hat{Q}_k^T {\mathbf{r}^k_{\text{proj}}} \rVert_2 \le  
\lVert (\hat{T}_k)^{-1} \rVert_2  \lVert \hat{Q}_k^T {\mathbf{r}^k_{\text{proj}}} \rVert_2 = \lVert (\hat{T}_k)^{-1} \rVert_2  \lVert {\mathbf{r}^k_{\text{proj}}} \rVert_2
\end{aligned}
\end{equation}
where the last equality is obtained using the property of orthogonal matrices acting as isometric transformations. 
{The property of the projection operator always ensures that $\lVert \mathbf{r}^k_{\text{proj}} \rVert_2 \leq \lVert \mathbf{r}^k \rVert_2 $, thus allowing the following step}
\begin{equation}
\lVert {\hat{\mathbf{g}}^k}\rVert_2 \le \lVert (\hat{T}_k)^{-1} \rVert_2  \lVert \mathbf{r}^k \rVert_2.
\end{equation}
Since $\lvert { \hat{g}_i^{(k)}} \rvert \le \lVert {\hat{\mathbf{g}}^k} \rVert_2$ for $i=1,\ldots,k,$ and
\begin{equation}
    \lVert (\hat{T}_k)^{-1}  \rVert_2 = \left({\sigma_{\min}(\hat{T}_k)}\right)^{-1}
\end{equation}
with $\sigma_{\min}(\hat{T}_k)$ the smallest non-zero singular value of $\hat{T}_k$, 
the lemma follows.
\end{proof}
\end{lemma}

Lemma \ref{lemma1} supports the demonstration of the following theorem, which provides a dynamic criterion for adjusting the accuracy of linear algebra operations that involve the matrix $\hat{R}_k$ in solving the least-squares, thus allowing for less expensive computations to solve \eqref{least_squares_perturbed_matrix}, while maintaining accuracy. 

\begin{theorem}
\label{thm1}
Let $\epsilon>0$.
% and $\varepsilon = \epsilon \lVert A \rVert_2$. 
Let $\mathbf{r}^{k}$ be the residual of AAR after $k$ iterations. Assume that $k={z}p$ iterations of Full AAR have been carried out. If for every $i\le k$,
\begin{equation}
    \lVert E_i \rVert_2 \le \frac{\sigma_{\min}(\hat{T}_k)}{k} 
    \frac{1}{\lVert \mathbf{r}^{k}\rVert_2}
    \frac{1}{\lVert \mathbf{x}^{i} - \mathbf{x}^{i-1} \rVert_2}\epsilon,
    \label{error_bound1}
\end{equation}
then $\delta_k \leq {\lVert A \rVert_2 \epsilon}$. 
\end{theorem}
\begin{proof}
The definition of $\delta_k$ in Equation \eqref{distance} and the upper bound for $\lvert {\hat{g}_i^{(k)}}\rvert$ from Lemma \ref{lemma1} lead to the following chain of inequalities
\begin{equation}
     \delta_k  \le \sum_{i=1}^k \lvert { \hat{g}_i^{(k)}}\rvert {\lVert A \rVert_2} \lVert E_i \rVert_2  \lVert \mathbf{x}^i-\mathbf{x}^{i-1}\rVert_2 \le 
    \sum_{i=1}^k \frac{1}{\sigma_{\min}(\hat{T}_k)} \lVert \mathbf{r}^k \rVert_2 { \lVert A \rVert_2} \lVert E_i \rVert_2  \lVert \mathbf{x}^i-\mathbf{x}^{i-1}\rVert_2.
\end{equation}
Therefore, using the assumption \eqref{error_bound1},
we obtain 
\begin{equation}
     \delta_k  \le 
    \sum_{i=1}^k \frac{1}{\sigma_{\min}(\hat{T}_k)} \lVert \mathbf{r}^k \rVert_2 { \lVert A \rVert_2} \lVert E_i \rVert_2  \lVert \mathbf{x}^i-\mathbf{x}^{i-1}\rVert_2  \le  \sum_{i=1}^k \frac{1}{k} \epsilon \lVert A \rVert_2.
    %= \varepsilon.
\end{equation}
\end{proof}
\begin{remark}
Lemma \ref{lemma1} and Theorem \ref{thm1} reproduce for Full AAR similar theoretical results \cite{simoncini} for GMRES and full orthogonalization method (FOM). 
Theorem \ref{thm1} states that one can afford performing more and more inaccurate evaluations of the residual throughout successive iterations and still ensure that Full AAR converges with the final residual below a prescribed threshold $\varepsilon$.
In particular, when the calculation of the residual is approximate due to the restriction of the residual onto a subspace, the theorem states that the restriction can be more and more aggressive as the residual norm decreases. 
\end{remark}

\begin{remark}
Theorem \ref{thm1} is of important theoretical value, but the upper bound for $\lVert E_i\rVert_2$ given in \eqref{error_bound1} is not easily checkable during the iterations for the following reasons: 
\begin{itemize}
\item the matrix $\hat{T}_k$ and the computed residual $\hat{\mathbf{r}}^k$ are not available at $i<k$ iterations.
\item $\hat{T}_k$ depends on the magnitude of the perturbations occurring in the approximate evaluations of the residual during $k$ steps.  
\item computing the minimum singular value $\sigma_{\min}(\hat{T}_k)$ is as expensive as the computation we want to avoid and therefore defeats the purpose.
\end{itemize}
\end{remark}
To address the first impracticality of the error bound in Equation \eqref{error_bound1}, we provide the following corollary that replaces $\mathbf{r}^k$ with $\mathbf{r}^i$, which is computable at iteration $i<k$. 
\begin{corollary}
\label{corollary1}
Let $\epsilon>0$.
%and $\varepsilon = \epsilon \lVert A \rVert$. 
Assume that $k={z}p$ iterations of Full AAR have been carried out. If for every $i\le k$,
\begin{equation}
    \lVert E_i \rVert_2 \le \frac{\sigma_{\min}(\hat{T}_k)}{k} 
    \frac{1}{\lVert \mathbf{r}^{i}\rVert_2}
    \frac{1}{\lVert \mathbf{x}^{i} - \mathbf{x}^{i-1} \rVert_2}\epsilon,
        \label{error_bound3}
\end{equation}
then $\delta_k \leq {\lVert A \rVert_2 \epsilon}$. 
%then $\delta_k \leq \varepsilon$. 
\end{corollary}
\begin{remark}
We point out that the error bound \eqref{error_bound3} for any iteration $i$ relies on the iteration index $k$. Convenient values of $k=k^*$ that can be chosen are the maximum number of iterations allowed or the size of the matrix. 
\end{remark}
\begin{remark}
The singular value $\sigma_{\min}(\hat{T}_k)$ could go to zero when the iterate converges to a solution since the columns may become linearly dependent. Projecting the residual matrix onto an appropriate subspace, as proposed later in Section~\ref{section:reduced_aar}, may help recover the linear independence of the columns and thus increase the minimum singular value of $\hat{T}_k$.
\end{remark}

The error bound in Equation \eqref{error_bound3} is still impractical because $\sigma_{\min}(\hat{T}_k)$, unlike $\mathbf{r}^i$, is not computable at iteration $i$, for $i<k$.  
We propose to replace $\sigma_{\min}(\hat{T}_k)$ in \eqref{error_bound3} with a practical heuristic bound. 
{A choice that uses 1 as the heuristic bound for $\sigma_{\min}(\hat{T}_k)$ has been discussed in the context of GMRES \cite{simoncini}}, where it is mentioned that this heuristic bound may be occasionally too far from the theoretical bound, {resulting in residuals higher than the prescribed threshold}. To address this limitation, it was proposed \cite{simoncini} to replace 1 with an inexpensive estimate that resorts to an approximation of the minimum non-zero singular value of the matrix $A$. In this work, we propose to combine the original heuristic bound that uses 1 at the numerator with a check on the monotonicity of the residual norm across consecutive AA steps. In fact, theoretical results \cite{lupo_pasini_convergence_2019} ensure that for exact calculations Full AAR produces a sequence of AA updates that monotonically decreases the residual under Assumption~\ref{FullrankAssumption}. If the approximate calculations cause the residual norm to stagnate between two consecutive AA updates, we use this as an indication that the last approximate AA step may have injected too large errors in the calculations. To attempt recovering the monotonic decrease, we thus propose to revert to the previous AA iterate and then take the AA step again with more accurate least-squares solutions. 
This consideration allows us to propose the following practical guideline for {adaptively} adjusting the accuracy of the calculations at each iteration, {while} still ensuring convergence of the numerical scheme within a prescribed residual tolerance.  
{

\begin{remark}
\label{heuristic}
Let us denote the appropriately adjusted heuristic error estimate as
\begin{equation}
{\beta}_i = \displaystyle \frac{\gamma_i}{k} 
    \frac{1}{\lVert \mathbf{r}^{i}\rVert_2}
    \frac{1}{\lVert \mathbf{x}^{i} - \mathbf{x}^{i-1}\rVert_2}, 
    \label{formula_gamma}
\end{equation}
where $\gamma_i>0$ for any iteration $i$ with $i\le k$.

\noindent We propose a guideline that adaptively adjusts the accuracy of the least-squares solves for the Anderson mixing at iteration $k = zp$ in the following steps:
\begin{enumerate}
\item Fix an initial level of accuracy in the least-squares calculations. 
\item If $\lVert E_i \rVert_2 \le {\beta}_i \epsilon$  for any $i$ with $i\le k$
\begin{enumerate}
    \item If $\lVert \mathbf{r}^{k}\rVert_2 < \lVert \mathbf{r}^{(z-1)p}\rVert_2$, then compute the Anderson mixing, update $\mathbf{x}^k$, and proceed to the next iteration;
    \item Else, decrease the value of $\gamma_i$, increase the accuracy of the least-squares calculations, and go back to step 2.
\end{enumerate}
\item Else, increase the accuracy of the least-squares calculations and go back to step 2.
\end{enumerate}
\end{remark}
The condition in step 2 aims to estimate the amount of inaccuracy that can be tolerated while performing the least-squares calculations at the current AA step. Since the theoretical results suggest that AAR can tolerate more inaccuracy in the least-squares calculations as converging to the solution, the condition in step 2 checks that the algorithm is not wastefully demanding more accuracy than what is necessary.  

The value of $\gamma_i$ may be larger than $\sigma_{\min}(\hat{T}_k)$, which implies that the rigorous theoretical bound may not be respected. However, if the existing level of accuracy still ensures a monotonic decrease of the norm of the residual in step 2(a), then the current iteration still benefits the numerical scheme towards convergence and can be accepted. If instead the existing level of accuracy compromises the monotonic decrease of the norm of the residual, it is recommended to start again the current iteration and solve again the same least-squares problem more accurately. 

The guideline provided above is sufficient to ensure the final convergence of AAR within a prescribed tolerance of the residual. When the numerical scheme never allows for inaccuracy in the least-squares calculations, the numerical scheme operates as Full AAR and the existing theoretical results \cite{lupo_pasini_convergence_2019} for Full AAR already ensure the Full AAR iterate eventually satisfies a prescribed residual tolerance. 
}

\subsection{Backward stability analysis of AAR with perturbed LHS and RHS}
\label{inaccuracy_both}
In this subsection, we consider the case that the mixing coefficients $\mathbf{g}^k$ in \eqref{anderson_correction} is approximated by the solution ${\hat{\mathbf{g}}^k}$ of the perturbed least-squares problem
\begin{equation}
  {\hat{\mathbf{g}}^k} =\underset{\mathbf{g}\in \mathbb{R}^{{k}}}{\operatorname{argmin}}\lVert  \hat{R}_k  \mathbf{g} - \hat{\mathbf{r}}^k  \rVert_2^2,
  \label{least_squares_perturbed_matrix_and_rhs}
\end{equation}
where $\hat{R}_k$ and $\hat{\mathbf{r}}^k$ are defined in \eqref{eq:perturbed_LHS} and \eqref{eq:perturbed_RHS}, respectively.
When both LHS and RHS of the least-squares are perturbed \cite{saunders}, the optimal backward error to solve the least-squares problem in Equation \eqref{least_squares_perturbed_matrix_and_rhs} at a generic iteration $k$ is
\begin{equation}
    \min_{\mathcal{E}_k, \delta\mathbf{r}^k} \{ \lVert \mathcal{E}_k, \delta\mathbf{r}^k \rVert_F \; \text{:} \; \lVert  (R_k + \mathcal{E}_k) {\hat{\mathbf{g}}^k} - (\mathbf{r}^k+\delta\mathbf{r}^k)  \rVert_2^2 = \min \}.
\end{equation}
The analysis conducted in this section assumes that the error made on the RHS of the least-squares problem can be tuned in a practical way. For instance, when the least-squares problem is projected onto a subspace, the distance between the original residual vector and the projected residual vector, namely $\lVert \delta \mathbf{r}^k \rVert_2$, can be measured in a straightforward way and the projection techniques can be adjusted accordingly. As for the error made on the LHS, namely $\delta_k$, this is less straightforward to control as it generally results from a sequence of several manipulations performed on the matrix $R_k$. The backward error analysis performed in this section gives the requirements on the inaccuracy of the operations performed on the matrix $R_k$, given a prescribed inaccuracy in the evaluation of the RHS. 
% {[TO BE REMOVED:]
% In this context, we define for convenience
% \begin{equation}
%     \tilde{\varepsilon} = \epsilon \max( \lVert A \rVert_2, \lVert \mathbf{r}^k \rVert_2)
% \end{equation}
% and we study the conditions required to ensure that $\delta_k \leq \tilde{\varepsilon}$. 
% [I don't see the $\lVert \mathbf{r}^k \rVert_2$ part in $\tilde{\varepsilon}$ is needed in the analysis (except for Corollary~2, so I replaced it with $\lVert A \rVert_2 \epsilon$ in Thm~2.]
% }

To obtain an upper bound on $|{\hat{g}_i^{(k)}}|$, we apply the analysis in Lemma~\ref{lemma1} to \eqref{least_squares_perturbed_matrix_and_rhs} and obtain, for $i= 1,\ldots,k$,
\begin{equation}
    \lvert {\hat{g}_i^{(k)}} \rvert \le  \frac{1}{\sigma_{\min}(\hat{T}_k)} \lVert \hat{\mathbf{r}}^k \rVert_2 \leq \frac{1}{\sigma_{\min}(\hat{T}_k)} \bigg ( \lVert \mathbf{r}^k \rVert_2 + \lVert \delta \mathbf{r}^k \rVert_2 \bigg),
\end{equation}
where the second inequality follows from \eqref{eq:perturbed_RHS} and the triangle inequality.
Assuming that $\|\delta \mathbf{r}^k\|_2\leq \epsilon\|\mathbf{r}^k \|_2$, we obtain 
\begin{equation}
    \lvert {\hat{g}_i^{(k)}} \rvert 
    \le \frac{1}{\sigma_{\min}(\hat{T}_k)} (1+\epsilon) \lVert \mathbf{r}^k \rVert_2,  \quad \text{for any} \quad i= 1,\ldots,k.
    \label{eq:g_bound}
\end{equation}
By following the same linear algebra steps applied in Theorem \ref{thm1} where we replace $\lVert \mathbf{r}^k \rVert_2$ with $\lVert \mathbf{r}^k \rVert_2 + \lVert \delta \mathbf{r}^k \rVert_2$, we obtain the following theoretical result.

\begin{theorem}
\label{thm2}
Let $\epsilon>0$.
% and let 
% ${\tilde{\varepsilon}} = \epsilon \max (\lVert \mathbf{r}^k \rVert_2 , \lVert A \rVert_2)$. 
Let $\hat{\mathbf{r}}^{k}$ be the residual of AAR after $k$ iterations. Under the same hypotheses and notation of Lemma \ref{lemma1} and assuming that $\|\delta \mathbf{r}^k\|_2\leq \epsilon\|\mathbf{r}^k \|_2$, if for every $i\le k$,
\begin{equation}
    \lVert E_i \rVert_2 \le \frac{\sigma_{\min}(\hat{T}_k)}{k} 
    \frac{1}{\lVert \mathbf{r}^{k} \rVert_2}
    \frac{1}{\lVert \mathbf{x}^{i} - \mathbf{x}^{i-1} \rVert_2}\frac{\epsilon}{1+\epsilon},
    \label{error_bound2}
\end{equation}
then $\delta_k \leq {\lVert A \rVert_2\epsilon}$. 
%then $\delta_k \leq \tilde\varepsilon$. 
\end{theorem}
\begin{proof}
The definition of $\delta_k$ in Equation \eqref{distance} and the upper bound for $| {\hat{g}_i^{(k)}}|$ in \eqref{eq:g_bound} lead to the following inequality 
\begin{equation}
\begin{aligned}
     \delta_k  &    \le  \sum_{i=1}^k \frac{1}{\sigma_{\min}(\hat{T}_k)} (1+\epsilon) \lVert \mathbf{r}^k \rVert_2 { \lVert A \rVert_2} \lVert E_i \rVert_2  \lVert \mathbf{x}^i-\mathbf{x}^{i-1}\rVert_2\:.
\end{aligned}
\end{equation}
Therefore, applying the assumption \eqref{error_bound2} to the right-hand side of above then gives
\begin{equation}
 \delta_k 
 \le \sum_{i=1}^k \frac{1}{k} (1+\epsilon) \frac{\epsilon}{(1+\epsilon)} \lVert A \rVert_2 
{ = \lVert A \rVert_2\epsilon}.
 % \le \tilde\varepsilon.
\end{equation}
\end{proof}

Theorem \ref{thm2} states that one can afford performing more and more inaccurate evaluations of the residual throughout successive iterations and still ensure that the norm of the backward error in Full AAR becomes smaller than a prescribed threshold $\tilde{\varepsilon}$. 
In particular, when the calculation of the residual is approximate due to the restriction of the residual onto a subspace, the theorem states that the restriction can be more and more aggressive as the residual norm decreases. 

\begin{remark}
Theorem \ref{thm2} gives a slightly more stringent requirement on $\lVert E_i \rVert_2$ than the one provided in Theorem \ref{thm1} to preserve backward stability when performing AA updates. This is reasonably expected because the hypotheses of Theorem \ref{thm2} assume that the perturbation affects both LHS and RHS.
\end{remark}

The considerations provided in section~\ref{inaccuracy_lhs} that motivate the need for a heuristic quantity to dynamically adjust the inaccuracy of the least-squares solve can be extended to this section as well.

{
In the analysis presented in Section~\ref{inaccuracy_both}, we provide conditions on $E_i$ and $\delta\mathbf{r}^k$ such that the backward error $\delta_k := \| \mathcal{E}_k \hat{\mathbf{g}}^k \|_2$ is below a prescribed threshold $\tilde{\varepsilon}$.
The following corollary bounds the effect on the AAR residual from a perturbed least-squares solve in terms of $\tilde{\varepsilon}$.
}
\begin{corollary}
\label{cor:forward}
    Under the same assumptions as in Theorem~\ref{thm2},
    let ${\overline{\mathbf{x}}}^k$ denote the results of an Anderson mixing according to Equation \eqref{anderson_correction} using the mixing vector $\mathbf{g}^k$ obtained by solving the exact least-squares problem in Equation \eqref{least_squares_matrix_form}. Let $\hat{\overline{\mathbf{x}}}^k$ denote the results of Anderson mixing according to Equation \eqref{anderson_correction} using the mixing vector $\hat{\mathbf{g}}^k$ obtained by solving the perturbed least-squares problem in Equation \eqref{least_squares_perturbed_matrix_and_rhs}. 
    Let ${\mathbf{x}}^{k+1}$ and $\hat{\mathbf{x}}^{k+1}$ be the updated iterates from ${\overline{\mathbf{x}}}^k$ and $\hat{\overline{\mathbf{x}}}^k$ respectively using \eqref{relaxation_step}. Then
    \begin{equation}
        \|\hat{\mathbf{r}}^{k+1} - \mathbf{r}^{k+1}\|_2 \leq {C (\sqrt{2} \kappa \|A\|_2 + \hat{\kappa}\|\mathbf{r}^k\|_2)\epsilon}
        % C (\sqrt{2}\kappa + \hat{\kappa})\tilde{\varepsilon},
    \end{equation}
    where $C$ is a constant that depends only on {$\|I - \omega A \|_2$}, $\kappa:=\frac{\sigma_{\max}(T_k)}{\sigma_{\min}(T_k)}$ is the condition number of $T_k$, and $\hat{\kappa}:=\frac{\sigma_{\max}(T_k)}{\sigma_{\min}(\hat{T}_k)}$.
\end{corollary}
\begin{proof}
    {
From the definition $\mathbf{r} = \mathbf{b} - A \mathbf{x}$ and \eqref{relaxation_step}, we have
\begin{equation}
    \|\hat{\mathbf{r}}^{k+1} - \mathbf{r}^{k+1}\|_2 
    = \| A (\hat{\mathbf{x}}^{k+1} - \mathbf{x}^{k+1})\|_2 
    = \| (A - \omega A^2)(\hat{\overline{\mathbf{x}}}^{k} - \overline{\mathbf{x}}^{k})\|_2 \:.
\end{equation}
It then follows from \eqref{anderson_correction} that
\begin{equation}
    \| (A - \omega A^2)(\hat{\overline{\mathbf{x}}}^{k} - \overline{\mathbf{x}}^{k})\|_2 
    =  \|(I - \omega A) R_k (\hat{\mathbf{g}}^{k} - \mathbf{g}^{k})\|_2\leq \|I - \omega A \|_2 \|R_k (\hat{\mathbf{g}}^{k} - \mathbf{g}^{k})\|_2.
\end{equation}
Since $\omega$ and $A$ are given, $\|I - \omega A \|_2\leq C$ for some $C>0$. Let ${{R}}_k^\dagger$ and $\hat{{R}}_k^\dagger$ denote the pseudo inverses of ${R}_k$ and $\hat{{R}}_k$, respectively. Then, the definitions of $\hat{\mathbf{g}}^{k}$, $\mathbf{g}^{k}$, and $\hat{\mathbf{r}}^{k}$ lead to
\begin{equation}
    \|R_k (\hat{\mathbf{g}}^{k} - \mathbf{g}^{k})\|_2 
=\|R_k (\hat{R}_k^\dagger \hat{\mathbf{r}}^k - {R}_k^\dagger {\mathbf{r}}^k)\|_2 
=\|R_k (\hat{R}_k^\dagger - {R}_k^\dagger) {\mathbf{r}}^k + R_k \hat{R}_k^\dagger \delta{\mathbf{r}}^k\|_2. 
\end{equation} 
We then obtain the desired bound via
\begin{equation}
\begin{alignedat}{2}
    \|\hat{\mathbf{r}}^{k+1} - \mathbf{r}^{k+1}\|_2 &\leq C \|R_k (\hat{\mathbf{g}}^{k} - \mathbf{g}^{k})\|_2 \\
&\leq C \big(\|R_k\|_2
\| (\hat{R}_k^\dagger - {R}_k^\dagger)\|_2
\|{\mathbf{r}}^k \|_2 + \|R_k\|_2 \|\hat{R}_k^\dagger \|_2 \|\delta{\mathbf{r}}^k\|_2\big)\\
&\leq C \big( \sqrt{2} \kappa (\sigma_{\min}(\hat{T}_k))^{-1}\|\mathcal{E}_k\|_2 \|\mathbf{r}^k\|_2 + \hat{\kappa}\|\delta\mathbf{r}^k\|_2 \big) \\
&\leq C \big( \sqrt{2} \kappa \|A\|_2 \frac{\epsilon}{1+\epsilon}  + \hat{\kappa} \epsilon \|\mathbf{r}^k\|_2 \big) 
\leq {C (\sqrt{2} \kappa \|A\|_2 + \hat{\kappa}\|\mathbf{r}^k\|_2){\epsilon}},
% \leq C (\sqrt{2} \kappa + \hat{\kappa})\tilde{\varepsilon},
\end{alignedat}
\end{equation}
where the third inequality uses the perturbation analysis of pseudo inverses \cite[Theorem~4.1]{wedin1973perturbation} and the relation between the induced 2-norms and singular values of matrices, and the fourth inequality follows from the bounds on $\|E_i\|_2$ and $\|\delta\mathbf{r}^k\|_2$ in Theorem~\ref{thm2}.
}
\end{proof}

\section{Reduced Alternating AA}
\label{section:reduced_aar}
Now we are ready to use the guidelines motivated from the theoretical results to propose a new variant of AA which reduces the computational effort by judiciously projecting the least-squares problem onto a subspace. We call this method \textit{Reduced Alternating AA}.
At iteration $k=z p$, Reduced Alternating AA computes approximate mixing coefficients $\hat{\mathbf{g}}^k$ by solving
\begin{equation}
\hat{\mathbf{g}}^k = \underset{\mathbf{g}\in \mathbb{R}^{m}}{\operatorname{argmin}}\lVert 
P_k (R_k\mathbf{g} - \mathbf{r}^k)\rVert^2_{2} \:,   
\label{eq:projected_LS}
\end{equation}
where $P_k\in\mathbb{R}^{n\times n}$ of rank $s<n$ is a projection operator onto an $s$-dimensional Euclidean subspace $\mathcal{S}_k\subset \mathbb{R}^n$.
To keep the computation cost of the projection low, we focus on the case in which the projection is restricted to a row selection procedure, i.e., 
\begin{equation}
    P_k = S_k S_k^T\quad\text{with}\quad S_k =[\mathbf{e}_{j_1}, \ldots, \mathbf{e}_{j_s}]\in \mathbb{R}^{n\times s}.
\end{equation}
Here $\mathbf{e}_{j}$ denotes the $j$-th canonical basis, and $S_k$ is the restriction operator that selects only the $j_i$-th rows, $i=1,\dots,s$.

In this paper, we propose the following two strategies to select the index set $\{j_i\}_{i=1}^s$.
\begin{itemize}
    \item {\it Subselected Alternating AA} -- the index set $\{j_i\}_{i=1}^s$ is selected to be the $s$ indices corresponding to the $s$ entries in the residual vector $\mathbf{r}^k$ with largest magnitudes.
    \item {\it Randomized Alternating AA} -- the index set $\{j_i\}_{i=1}^s$ is drawn from the full index set $\{1,\dots,n\}$ under a uniform distribution without replacement. 
\end{itemize}
With the same size $s$, the two strategies lead to similar reduction in terms of the computation time for solving \eqref{eq:projected_LS}, with a minor difference in the computation overhead in the index selection process. 
The dimensionality $s$ of the projection subspace $\mathcal{S}_k$ is adaptively tuned at each according to values of the heuristic that estimates the error bound in Equation \eqref{error_bound2}.

\subsection{Computational cost of Reduced Alternating AA for linear fixed-point iterations}
The computational cost to solve the least-squares reduces from $\mathcal{O}(nm^2)$ to $\mathcal{O}( s m^2)$.
Therefore, the computational complexity of Reduced AAR is less than the computational cost of AAR
\begin{equation}
    \mathcal{O}\bigg( {\texttt{nnz}(A)\,} n + \frac{1}{p} s m^2\bigg) \le \mathcal{O}\bigg( {\texttt{nnz}(A)\,} n + \frac{1}{p} n m^2\bigg).
\end{equation}
While the dimension $s$ of the subspace $\mathcal{S}_k$ is tuned according to the heuristic proposed in Equation \eqref{error_bound2}, the value of $p$ still needs to be empirically and judiciously tuned. 

\section{Numerical results}
\label{section:numerical_results}

In this section we present numerical results {to illustrate that AA with approximate calculations converges when the accuracy of approximate calculations is guided by the proposed heuristics. We applied the proposed solvers to example fixed-point problems, including a variety of linear systems considered in Section~\ref{subsec:linear} and a non-linear fixed-point problem from a simplified Boltzmann equation detailed in Section~\ref{subsec:nonlinear}.
We note that, even though the analysis in Section~\ref{section:error_bounds} only applies to the linear case, the variants of AA with heuristic-guided approximate calculations converge for both linear and non-linear fixed-point problems considered in this section. 
}
% where fixed-point iterations are performed inaccurately, but the AA still succeeds in converging by bringing the residual of the fixed-point iteration below a prescribed tolerance. We focus on two types of fixed-point iterations: linear deterministic and non-linear deterministic. 
% The numerical results for the linear deterministic fixed-point iteration describe the use of AA to accelerate the Richardson scheme for solving linear systems. The numerical results for the non-linear deterministic fixed-point iteration describe the use of AA to accelerate the solution of an implicit system from a simplified Boltzmann equation at a given time-step.

\subsection{Linear deterministic fixed-point iteration: iterative linear solver}
\label{subsec:linear}
The numerical results used to illustrate AA on linear deterministic fixed-point iterations to solve linear systems are split into two parts. First, we study how the convergence of AR is affected by random perturbation with tunable intensity applied to the LHS of the least-squares solved to compute the Anderson mixing. Then, we show how restricting the least-squares onto a projection subspace reduces the computational time to solve a linear system while still converging.    

\subsubsection{Injection of random noise on LHS of least-squares for each AA step}
We consider the diagonal matrix $A = \text{diag}(10^{-4}, 2, 3, \ldots, 100)$. {The exact solution of the linear system is chosen as a random vector where all the entries follow a uniform distribution in the range $[0,1]$, and the RHS is obtained as a result of the matrix-vector multiplication between the matrix $A$ and the solution vector.}
We add random noise {to} the LHS of the least-squares to compute the Anderson mixing:
\begin{equation}
    \hat{\mathbf{g}}^k=\underset{\mathbf{g}\in \mathbb{R}^{k}}{\operatorname{argmin}}{\bigg \lVert \bigg (R_k + \epsilon_k \lVert R_k\rVert_2 \hat{\mathcal{E}}_k \bigg)\mathbf{g} - \mathbf{r}^k \bigg \rVert_2^2}.
\end{equation}
The perturbation matrices $ \hat{\mathcal{E}}_k$'s, with $\lVert \hat{\mathcal{E}}_k\rVert_2=1$ for $k>0$ are random $100 \times 100$ matrices generated with normally distributed values using the \texttt{Matlab} function \textbf{randn}. The quantity $\epsilon_k$ is defined as 
\begin{equation}
    \epsilon_k =  \frac{\epsilon}{k^*} \frac{\sigma_{\min} (T_k)}{\lVert \mathbf{r}^k \rVert_2 \lVert \mathbf{x}^k - \mathbf{x}^{k-1} \rVert_2 },
    \label{test_error}
\end{equation}
where $T_k$ is the triangular factor of the QR factorization of $R_k$. 
The quantity $\epsilon$ is used to tune the magnitude of the noise $\epsilon_k$. We set $k^* = 100$ according to the size of the linear system we are solving. We use four different values of $\epsilon$: $\epsilon = 1e-8$, $\epsilon = 1e-6$, $\epsilon = 1e-4$. $\epsilon = 1.0$. 
Numerical results in Figure~\ref{random_noise} show $\epsilon_k$ as a function of the iteration count. The residual norm history of AR for both the unperturbed least-squares and the perturbed least-squares is shown as well to illustrate how the random perturbation affects the convergence of AR for each value of $\epsilon$ tested. As expected, increasing values of $\epsilon$ lead to a progressive deterioration of the convergence rate, but final convergence within the prescribed tolerance of the residual is still preserved. 

\begin{figure}
  \centering
  \subfloat[]{\label{profile_1e8}\includegraphics[width=0.5\textwidth]{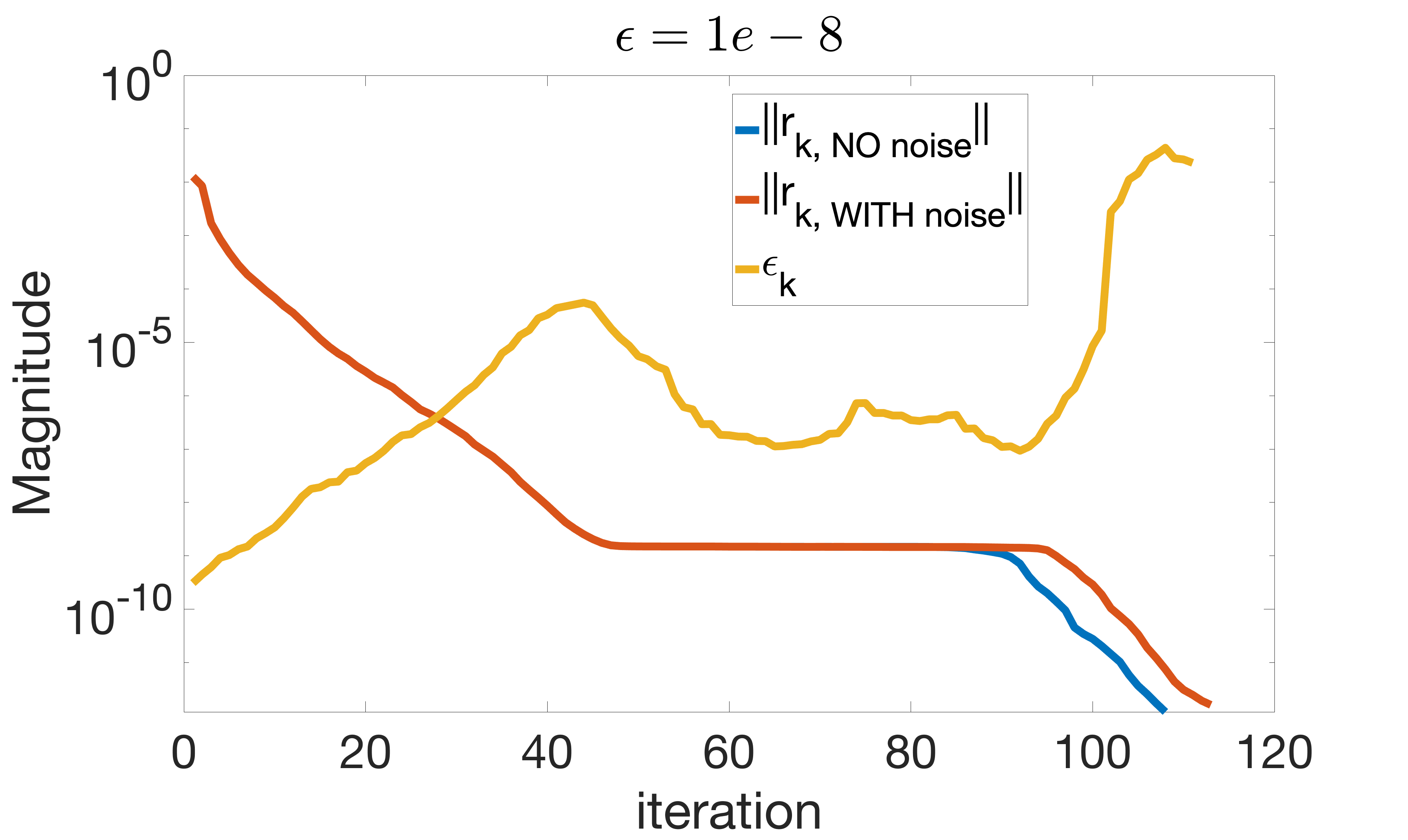}}
  \subfloat[]{\label{profile_1e6}\includegraphics[width=0.5\textwidth]{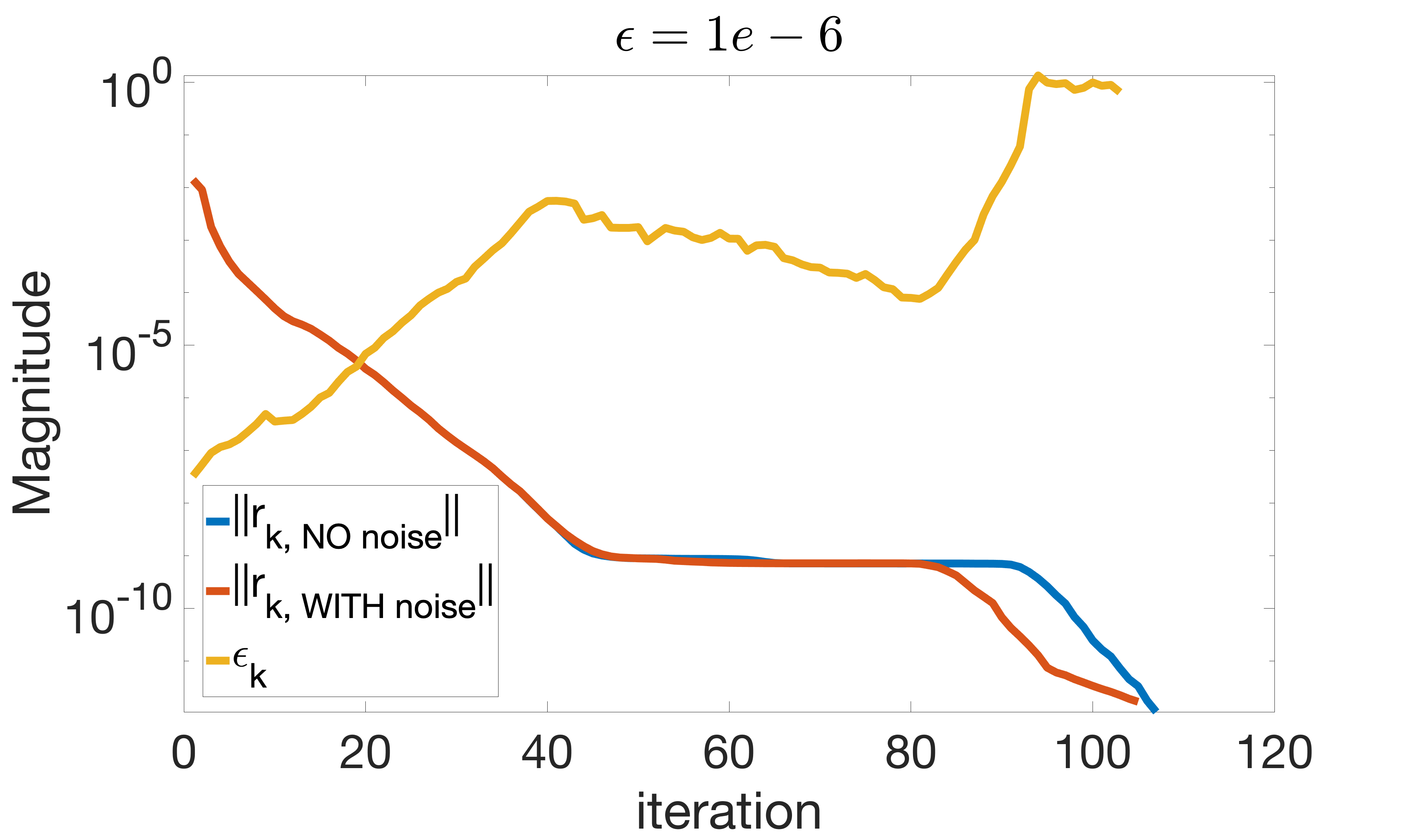}} 
  \newline
  \subfloat[]{\label{profile_1e4}\includegraphics[width=0.5\textwidth]{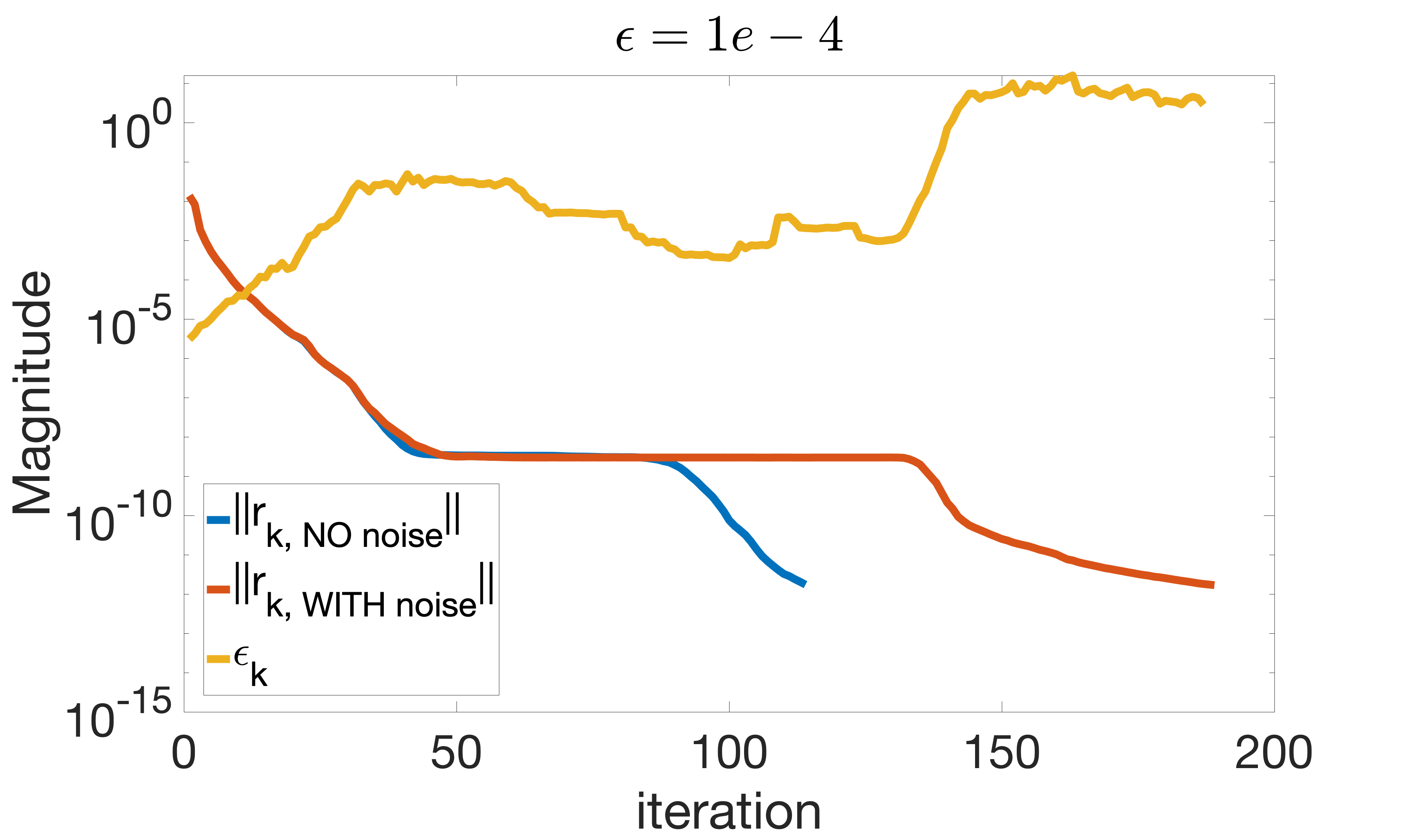}}
  \subfloat[]{\label{profile_1e0}\includegraphics[width=0.5\textwidth]{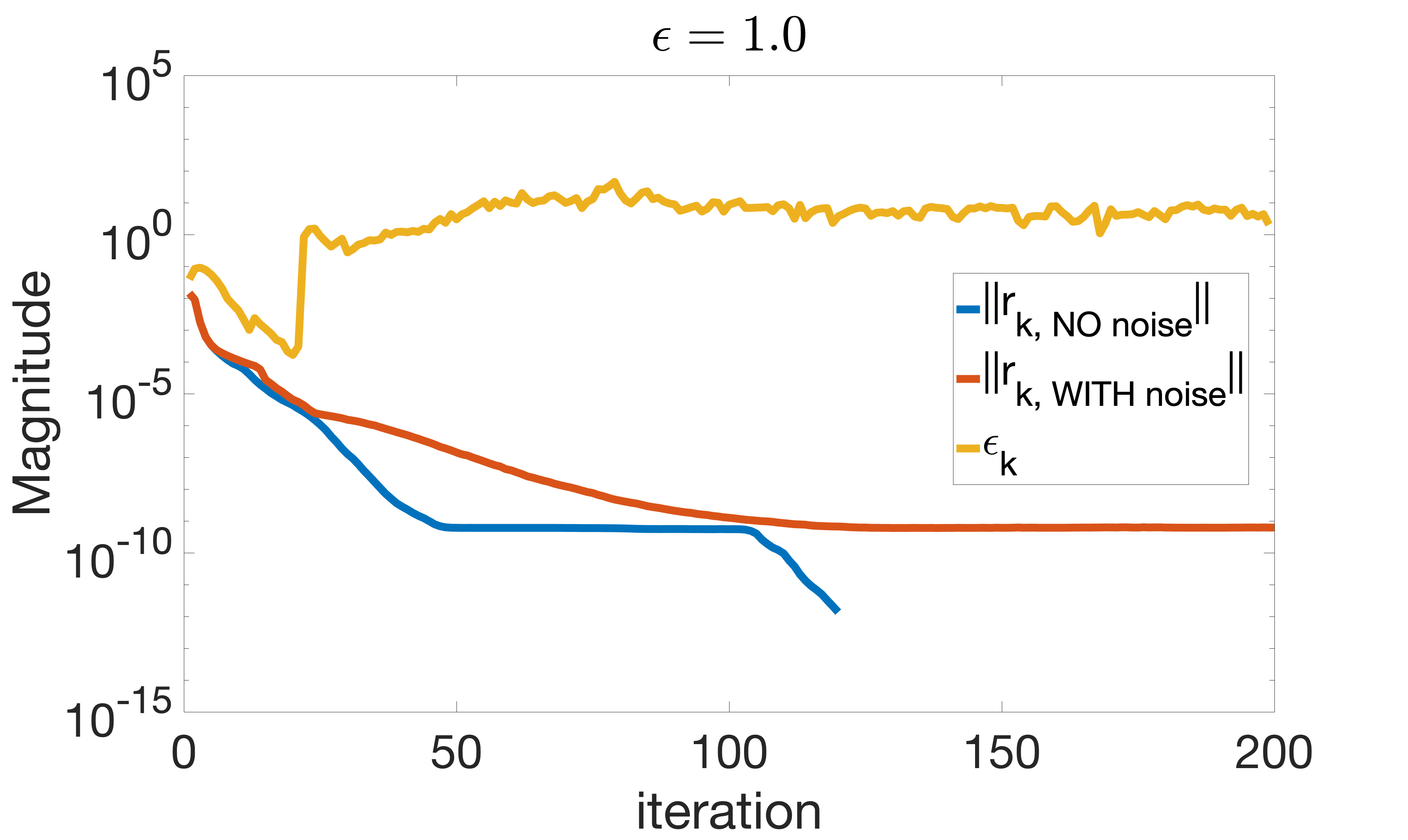}}  
  \caption{\label{random_noise}Approximate AA convergence history.}
\end{figure}

\subsubsection{Reduced Alternating AA to solve sparse linear systems}

We now use Reduced Alternating AA to solve a set of linear systems where the matrices are open source and pulled from the SuiteSparse Matrix Collection (formerly known as the University of Florida Sparse Matrix Collection) \cite{UFL}, the Matrix Market Collection \cite{MM}. 
In Table \ref{matrices}, we report the matrices and their most significant properties. The sources used to retrieve the matrices are specified in Table \ref{matrices} as well. The notation MM is used to refer to the Matrix Market Collection, and SS for the SuiteSparse Matrix Collection. 

\begin{table}
\center
\begin{tabular}{|c|c|c|c|c|c|}
    \hline
        \textbf{Matrix} & \textbf{Type} & \textbf{Size} & 
         \textbf{Structure} & \textbf{Pos. def.} & \textbf{Source}\\
    \hline  
    fidap029 & real & 2,870 & nonsymmetric & yes & MM\\
    \hline
    raefsky5 & real & 6,316 & nonsymmetric & yes & SS\\
    \hline
    bcsstk29* & real & 13,992 & symmetric & no & SS\\    
    \hline
    sherman3* & real & 5,005 & nonsymmetric & no & MM \\
    \hline
    sherman5* & real & 3,312 & nonsymmetric & no & MM\\
    \hline
    chipcool0 & real & 20,082 & nonsymmetric & no & SS\\
    \hline
    e20r0000* & real & 4,241 & nonsymmetric & no & MM \\
    \hline
    spmsrtls & real & 29,995 & nonsymmetric &no & SS \\
    \hline
    garon1 & real & 3,175 & nonsymmetric & no & SS \\
    \hline
    garon2 & real & 13,535& nonsymmetric & no & SS\\
    \hline
    memplus & real & 17,758 & nonsymmetric & no & SS\\
    \hline
    saylr4 & real & 3,564 & nonsymmetric & no & MM \\     
    \hline
    xenon1* & real & 48,600 & nonsymmetric & no & SS \\
    \hline
    xenon2* & real & 157,464 & nonsymmetric & no & SS\\
    \hline
    venkat01 & real & 62,424 & nonsymmetric & no & SS \\
    \hline
    QC2534 & complex & 2,534 & non-Hermitian & no & SS \\
    \hline
    mplate* & complex & 5,962 & non-Hermitian & no & SS \\
    \hline
    light\_in\_tissue & complex & 29,282 & non-Hermitian & no & SS \\
    \hline
    kim1 & complex & 38,415 & non-Hermitian & no & SS \\
    \hline
    chevron2 & complex & 90,249 & non-Hermitian & no & SS \\   
    \hline    
\end{tabular}
\caption{List of matrices used for numerical experiments performed in \sc{Matlab}.}
\label{matrices}
\end{table}

The experiments compare the performance of GMRES with restart parameter equal to 50, Alternating AA, Subselected Alternating AA and Randomized Alternating AA {in solving linear systems with the matrices described in Table \ref{matrices}. The exact solution of each linear system is chosen as a random vector where all the entries follow a uniform distribution in the range $[0,1]$, and the RHS is obtained as a result of the matrix-vector multiplication between the matrix and the associated solution vector}. For each problem, two preconditioning techniques are used: Zero Fill-In Incomplete LU Factorization (ILU(0) for short) and Incomplete LU Factorization with Threshold (ILUT($\tau$) for short) \cite{Chan1997}.
When ILUT($\tau$) is used, the tolerance $\tau$ controls the level of sparsity in the incomplete LU factors. We set $\tau=10^{-4}$ and any zeros on the diagonal of the upper triangular factor are replaced by the local tolerance. The ILUT($\tau$) preconditioner is computed with a column pivoting and the pivot is chosen as the maximum magnitude entry in the column.
For more details about ILU(0) and ILUT($\tau$) we refer to literature \cite{Benzi}, \cite[pp.~287--307]{Saad}. %Results with ILU(0) and ILUT($\tau$) are reported only if the algorithm to construct the preconditioner is numerically stable. 
A symmetric reverse Cuthill-McKee reordering \cite{rcm} has been applied to some matrices to allow a stable construction of the ILU factors.
The matrices that needed a reordering for the ILU preconditioner are marked with an asterisk close to their name in Table \ref{matrices}. 
%Lines of a table full of dashes signify that the linear solver has not reached convergence, whereas
%lines of a table full of asterisks signify that a stable construction of the preconditioner was not possible, even after a symmetric reverse Cuthill-McKee reordering of the coefficient matrix $A$. 
Since the matrices \texttt{xenon1} and \texttt{xenon2} were poorly scaled, the ILU factors have been computed for these matrices
only after a diagonal scaling. The diagonal scaling has been applied via a diagonal preconditioner on these matrices. Those situations where the matrices had zeros on the main diagonal were treated by setting to one the associated entries in the diagonal preconditioner.
The Richardson relaxation parameter is set to $\omega = 0.2$.  The number of Richardson iterations between two Anderson updates for the Alternating AA is $p=3$. 
{Two different settings for the history of AA are considered: in one case the full history is used without truncation ($m=k$) and in the other case the history is truncated to maintain fixed length equal to $m=20$.
The former setting is consistent to the no-truncation assumption we made in the theoretical analysis in Section~\ref{section:error_bounds}, whereas the latter setting is often used in practice for its preferable performance.} 
The least-squares problems related to Anderson mixing steps
are solved via QR factorization with column pivoting of the rectangular matrix to the LHS. 
The RHS of the linear systems is obtained by multiplying the solution vector by the coefficient matrix $A$. The variation across single runs never exceeded $10\%$ in terms of computational time. 
A threshold of $10^{-8}$ is used as a stopping criterion on the relative residual $\ell^2$-norm.
The heuristic proposed in Corollary \ref{heuristic} is used to automatically tune the dimensionality of the projection subspace at each AA step. To avoid introducing excessive computational overhead in fine tuning the dimensionality of the subspace, we add rows in batches, with each batch being 10\% of the total number of rows in the original least-squares problem. 

The performance of the linear solvers is evaluated through \textit{performance profiles} \cite{Dolan}. %These graphic tools provide an immediate visual approach to compare the performance of multiple algorithms tested on a set of benchmark problems. In order to explain how this performance profiles are generated, 
Let us refer to $\mathcal{S}$ as the set of solvers and $\mathcal{P}$ as the test set. We assume that we have $n_s$ solvers and $n_p$ problems. In this paper, performance profiles are used to compare the computational times. To this end we introduce
\[
t_{p,s}=\text{computing time to solve problem}\; p \; \text{with solver}\; s.
\]
The comparison between the times taken by each solver is based on the performance ratio defined as
\[
r_{p,s} = \frac{t_{p,s}}{\min\{t_{p,s}:s\in\mathcal{S}\}}.
\]
The performance ratio allows one to compare the performance of solver $s$ on problem $p$ with the best performance by any solver to address the same problem $p$. In case a specific solver $s$ does not succeed in solving problem $p$, then a convention is adopted to set $r_{p,s}=r_{M}$ where $r_M$ is a maximal value. In our case we set $r_m=10,000$. 
The performance of one solver compared to the others' on the whole benchmark set is displayed by the cumulative distribution function $\rho_s(\tau)$ that is defined as follows:
\[
\rho_s(\tau)=\frac{1}{n_p}\text{size}\{p\in\mathcal{P}:r_{p,s}\le \tau \}.
\]
The value $\rho_s(\tau)$ represents the probability that solvers $s\in\mathcal{S}$ has a performance ratio $r_{p,s}$ less than or equal to the best possible ratio up to a scaling factor $\tau$. 
%The cumulative distribution function is non-decreasing, piece-wise constant and continuous from the right at each discontinuity. A particular interpretation is associated with the value $\rho_s(1)$. In fact this value represents the probability that solver $s$ outperforms every other solver from set $\mathcal{S}$ in solving a generic problem from set $\mathcal{P}$. 
The convention adopted that prescribes $r_{p,s}=r_M$ if solver $s$ does not solve problem $p$ leads to the reasonable assumption that
\[
r_{p,s}\in[1,r_M].
\]
Therefore, $\rho_s(r_M)=1$ and 
\[
\rho_s^* = \lim_{\tau\rightarrow r_M^-}\rho_s(\tau)
\]
represents the probability that solver $s$ succeed in solving a generic problem from set $\mathcal{P}$. 
%In general, solvers with larger values of $\rho_s(\tau)$ are to be preferred.It may happen that some solvers from set $\mathcal{S}$ take a considerable amount of time in solving specific problems from set $\mathcal{P}$. This consequently requires to pick a sufficiently high value for $r_M$ and to extend the range of values of $\tau$ displayed in the performance profiles. This may lead to graphs that are difficult to interpret, since most of the main features of the curves may be shrunk on the far left of the graph window. Moreover, most of the window may be occupied to describe the trend of the curves for high values of $\tau$ where nothing relevant happens. 
To improve the clarity of the graphs, we replace $\rho_s(\tau)$ with
\begin{equation}\label{eq:performance_metric}
\tau \mapsto \frac{1}{n_p}\text{size}\{p\in\mathcal{P}:\log_2(r_{p,s})\le \tau\}    
\end{equation}
%Although using the log scale complicates the interpretation of the graph, it dedicates most of the space to values of $\tau$ where significant trends are captured and worth being discussed. 
and we use this quantity for each performance profile displayed.\newline

{The numerical results with ILU(0) and ILUT($\tau$) when AA uses full history ($m=k$) are shown in Figures \ref{profile_ilu0} and \ref{profile_ilut}, respectively. Randomized Alternating AA$(m=k)$ outperforms standard Alternating AA$(m=k)$ and Subselected Alternating AA$(m=k)$ for time-to-solution, showing that a randomized selection of the rows for the least-squares problem results in improved speed of the AA solver without compromising convergence to the desired accuracy. However, all variants of AA take longer to converge than Restarted GMRES for two main reasons. First, the matrices $R_k$ associated with the least-squares problems solved in AA are unstructured and dense, whereas the matrices associated with the least-squares problems solved in GMRES are upper Hessenberg matrices, which reduces the computational cost to solve the least-squares problems. Secondly, the computational cost to solve the least-squares problems for all variants of AA increases when $m=k$ due to the increase of columns in the matrix  $R_k$ at each iteration $k$.
We next explore the numerical performance of variants of AA when the history size $m$ is restricted to $m=20$.}
%{To provide an insight into the performance of the different variants of AA when the history of AA is truncated to limit the computational cost of each iteration, we performed the same study for $m=20$}. 
The numerical results with ILU(0) and ILUT($\tau$) {when the AA history is truncated to $m=20$} are shown in Figures \ref{profile_ilu0_m20} and \ref{profile_ilut_m20}, respectively. Randomized Alternating AA{$(m=20)$} outperforms Subselected Alternating AA{$(m=20)$} by converging on a larger set of problems, and outperforms standard Alternating AA{$(m=20)$} and GMRES for time-to-solution. Therefore, an adaptive reduction of the dimensionality of the least-squares problem is shown to effectively reduce the total computational cost of Alternating AA to iteratively solve a number of different sparse linear systems while still maintaining convergence.

\begin{figure}
  \centering
  \subfloat[]{\label{profile_ilu0}\includegraphics[width=0.5\textwidth]{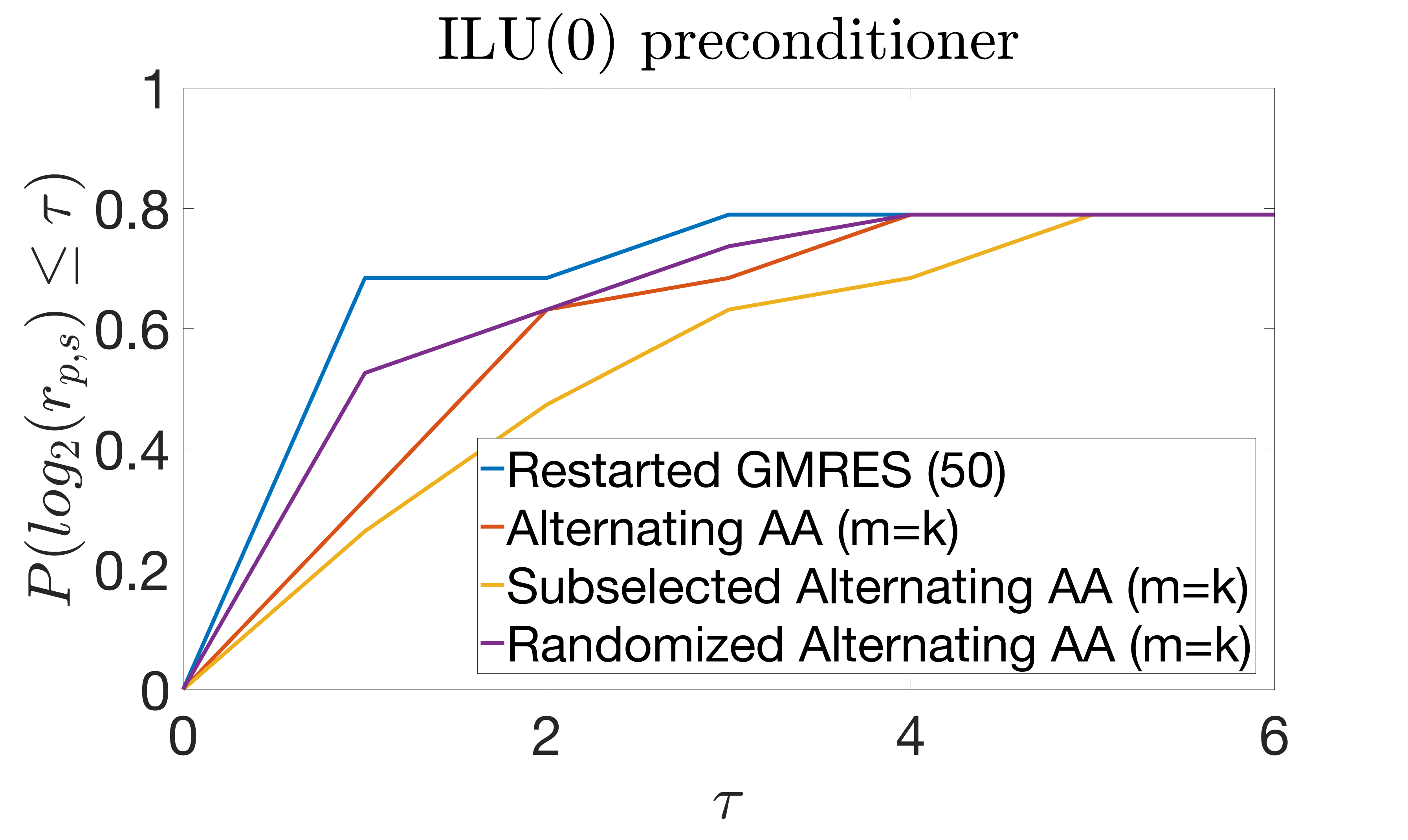}}
  \subfloat[]{\label{profile_ilut}\includegraphics[width=0.5\textwidth]{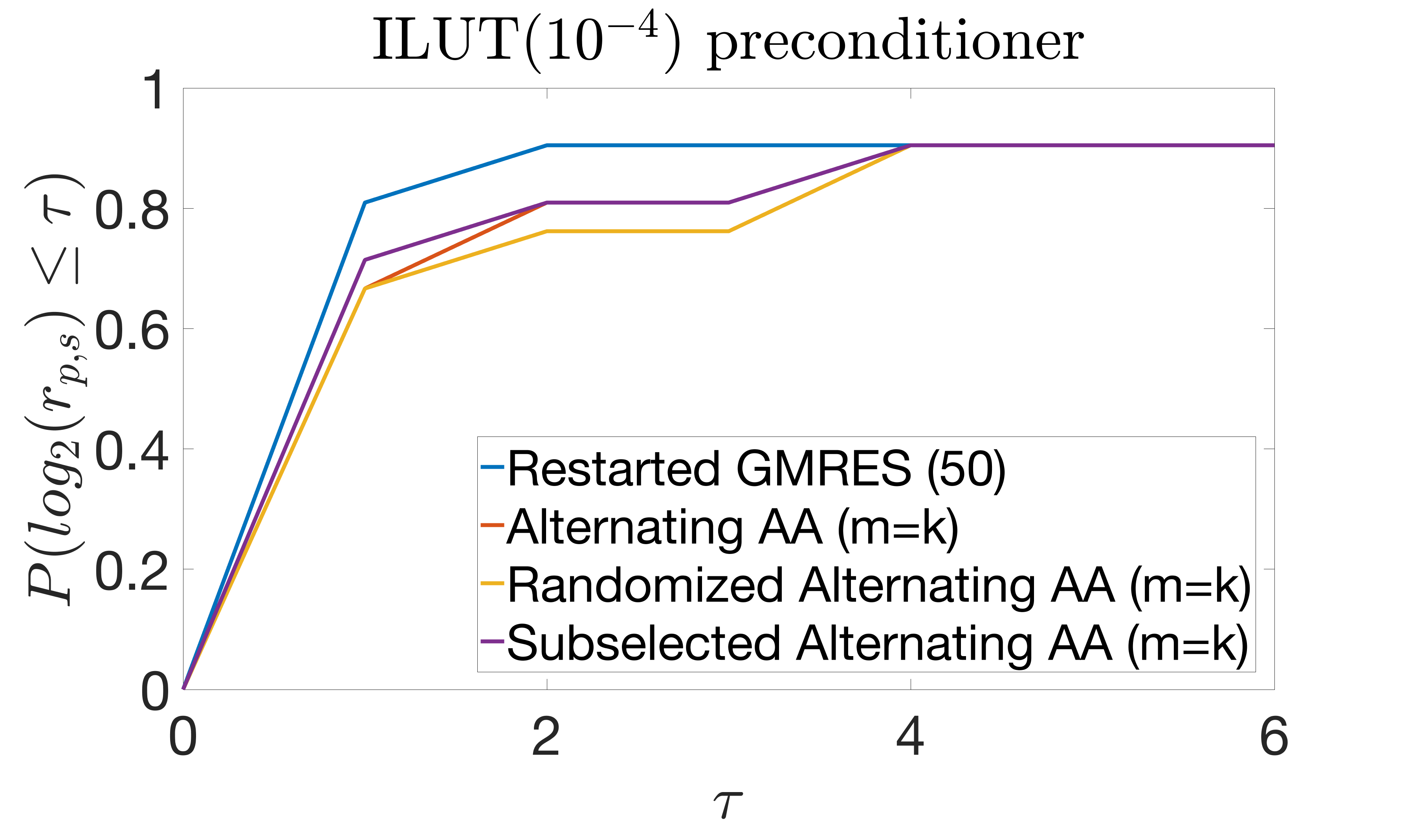}}  
  % \caption{\label{profiles}Performance profiles in a $\log_2$ scale {when the full history is used for AA variants ($m=k$)}.}
  \caption{\label{profiles_m=k}{Performance comparison between iterative solvers on sparse linear systems. Here the performance metric $P$ given in \eqref{eq:performance_metric} is reported for each solver $s$ with varying $\tau$. All AA variants use the full history ($m=k$)}.}
\end{figure}

\begin{figure}
  \centering
  \subfloat[]{\label{profile_ilu0_m20}\includegraphics[width=0.5\textwidth]{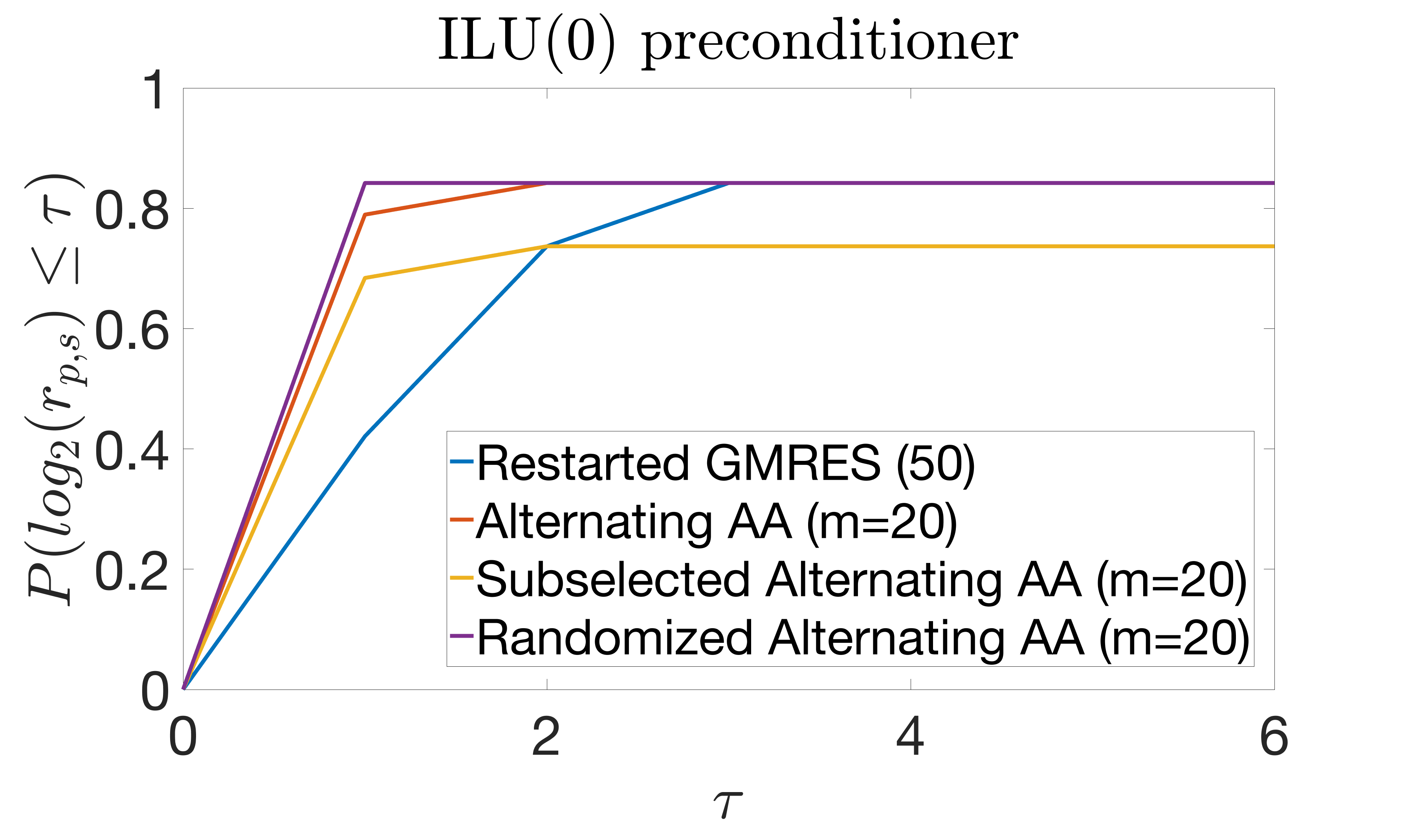}}
  \subfloat[]{\label{profile_ilut_m20}\includegraphics[width=0.5\textwidth]{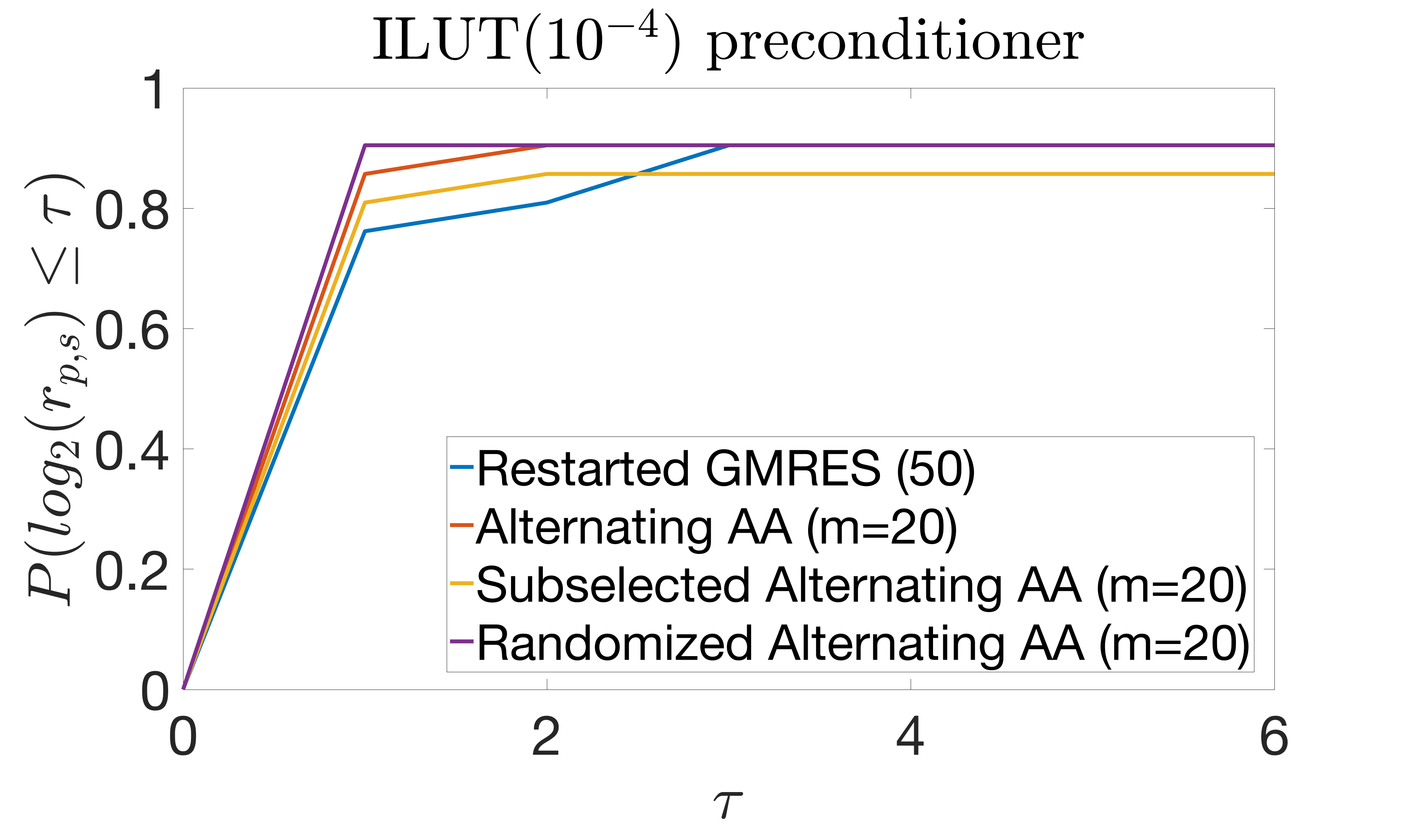}}  
  % \caption{\label{profiles}Performance profiles in a $\log_2$ scale {when AA history is truncated to $m=20$}.}
  \caption{\label{profiles}{Performance comparison between iterative solvers on sparse linear systems. Here the performance metric $P$ given in \eqref{eq:performance_metric} is reported for each solver $s$ with varying $\tau$. All AA variants use a history truncated to $m=20$}.}
\end{figure}

{
In Tables \ref{ls_vs_richardson_cost} and \ref{ls_vs_richardson_cost_m20}, we report the ratio between the total wall-clock time for solving the least-squares problems and the total wall-clock time for applying Richardson's steps in Alternating AA, Subselected Alternating AA, and Randomized Alternating AA, using ILU(0) and ILUT($\tau$) preconditioning methods for settings $m=k$ and $m=20$, respectively. 
We notice that, in Table~\ref{ls_vs_richardson_cost}, the computational cost to solve the least-squares problems is considerably higher than the cost to perform a Richardson's steps across all numerical tests, whereas 
% In Table \ref{ls_vs_richardson_cost_m20} we report the ratio between the total wall-clock time for solving the least-squares problems and the total wall-clock time for applying Richardson's steps in Alternating AA, Subselected Alternating AA, and Randomized Alternating AA, using ILU(0) and ILUT($\tau$) preconditioning methods {when $m=20$}. Compared to the numbers in Table \ref{ls_vs_richardson_cost} for $m=k$, we notice that 
the ratio reported in Table~\ref{ls_vs_richardson_cost_m20} is significantly lower, because the matrices $R_k$ there are tall and skinny and have a fixed number of columns across iterations. However, even in truncated history case in Table~\ref{ls_vs_richardson_cost_m20},
when using Alternating AA with the ILU(0) preconditioner}, solving the least-squares problems still takes longer than performing the Richardson's steps for all matrices except for the matrices bcsstk29 and QC2534. These results validate the claim that the computational time to solve the least-squares problems is generally higher than the computational time to perform Richardson's steps in Alternating AA. 
When the ILUT($\tau$) preconditioner is used, the discrepancy between the computational time to solve the least-squares problems and the computational time to perform the Richardson's steps reduces because the factors of the ILUT($\tau$) preconditioner have more non-zero entries than the ILU(0) preconditioner, thus making the preconditioner more computationally expensive to be applied at each Richardson's step. 
Both Subselected Alternating AA and Randomized Alternating AA effectively reduce the computational time to solve the least-squares problem by projecting the least-squares problem in a reduced space. The reduction of dimensionality of the least-squares problem precludes Subselected Alternating AA from reaching convergence for some matrices on which standard Alternating AA converged. On the contrary, Randomized Alternating AA succeeds in reducing the computational time to solve the least-squares problems without affecting the convergence.

\begin{table}

\center
{
\begin{tabular}{|c||p{1.8cm}|p{1.8cm}||p{1.8cm}|p{1.8cm}||p{1.8cm}|p{1.8cm}|}
    \hline
         &  \multicolumn{6}{|c|}{\Large $\frac{\textbf{Total wall-clock time for least-squares}}{\textbf{Total wall-clock time for Richardson's steps}}$}\\
    \hline
    & \multicolumn{2}{|c|}{\bf Alternating AA{$(m=k)$}} & \multicolumn{2}{|c|}{\bf Subselect. Altern. AA{$(m=k)$}} & \multicolumn{2}{|c|}{\bf Random. Altern. AA{$(m=k)$}}\\
    \hline
        \textbf{Matrix} & \textbf{ILU(0)} & \textbf{ILUT($\tau$)} & \textbf{ILU(0)} & \textbf{ILUT($\tau$)} & \textbf{ILU(0)} & \textbf{ILUT($\tau$)} \\       
    \hline  
    fidap029 & 13.71 & 39.85 & 4.09 & 12.01 & 3.72 & 3.95\\
    \hline
    raefsky5 & 23.44 & 33.14 & 6.18 & 9.21 & 5.80 & 9.37\\
    \hline
    bcsstk29 & 2.02 & 2.17 & 2.23 & 1.81 & 1.49 & 1.21\\    
    \hline
    sherman3 & 14.30 & 11.17 & 9.84 & 3.71 & 9.45 & 3.16 \\
    \hline
    sherman5 & 36.98 & 35.79 & 9.1 & 10.07 & 8.19 & 13.23\\
    \hline
    chipcool0 & 33.74 & 8.41 & 9.53 & 2.60 & 6.34 & 2.18\\
    \hline
    e20r0000 & 24.77 & 7.37 & 12.96 & 3.11 & 15.33 & 3.02\\
    \hline
    spmsrtls & 54.83 & 72.02 & 10.25 & 15.04 & 8.80 & 13.51\\
    \hline
    garon1 & 21.95 & 6.09 & 17.79 & 3.07 & 2.67 & 3.01\\
    \hline
    garon2 & 29.29 & 4.28 & 30.01 & 1.67 & 6.64 & 1.61\\
    \hline
    memplus & - & 44.05 & - & 9.81 & - & 7.41\\
    \hline
    saylr4 & 38.31 & 33.04 & 12.17 & 10.43 & 10.19 & 6.56\\     
    \hline
    xenon1 & - & 9.39 & - & 2.94 & - & 2.76 \\
    \hline
    xenon2 & - & 10.02 & - & 2.19 & - & 1.91 \\
    \hline
    venkat01 & 30.61 & 9.26 & 6.61 & 2.71 & 5.16 & 2.10\\
    \hline
    QC2534 & 5.94 & 7.14 & 5.79 & 3.48 & 5.71 & 3.42\\
    \hline
    mplate & - & 7.15 & - & 7.61 & - & 6.11 \\
    \hline
    light\_in\_tissue & 47.23 & 20.36 & 10.29 & 4.30 & 10.22 & 4.30\\
    \hline
    kim1 & 37.23 & 31.92 & 8.40 & 6.62 & 8.97 & 6.21\\
    \hline
    chevron2 & 66.01 & 18.60 & 72.47 & 4.64 & 47.41 & 4.02 \\   
    \hline    
\end{tabular}
}
\caption{{Ratio between total wall-clock time to solve least-squares problems and to perform Richardson's steps within a complete run of Alternating AA {$(m=k)$}, Subselected Alternating AA {$(m=k)$}, and Randomized Alternating AA {$(m=k)$} when ILU(0) and ILUT($\tau$) preconditioners are used. The dashed line indicates examples where the tested solver did not converge to the desired accuracy.}}
\label{ls_vs_richardson_cost}
\end{table}

%\begin{table}
%\center
%\begin{tabular}{|c|p{3.5cm}|p{3.5cm}|}
%    \hline
%         &  \multicolumn{2}{|c|}{\Large $\frac{\textbf{Total wall-clock time for least-squares}}{\textbf{Total wall-clock time for Richardson's steps}}$}\\
%    \hline
%        \textbf{Matrix} & \textbf{ILU(0)} & \textbf{ILUT($\tau$)}\\       
%    \hline  
%    fidap029 & 4.20 & 1.37 \\
%    \hline
%    raefsky5 & 3.13 & 1.12 \\
%    \hline
%    bcsstk29* & 1.03 & 1.01 \\    
%    \hline
%    sherman3* & 6.40 & 1.09 \\
%    \hline
%    sherman5* & 3.49 & 1.11 \\
%    \hline
%    fidap008* & 5.39 & 2.04 \\
%    \hline
%    chipcool0 & 6.85 & 2.26 \\
%    \hline
%    e20r0000* & 7.18 & 2.66 \\
%    \hline
%    spmsrtls & 2.67 & 2.45 \\
%    \hline
%    garon1* & 4.86 & 2.58 \\
%    \hline
%    garon2* & 4.82 & 2.71  \\
%    \hline
%    memplus & 15.15 & 1.77 \\
%    \hline
%    saylr4 & 9.87 & 9.82 \\     
%    \hline
%    xenon1* & 11.62 & 0.85 \\
%    \hline
%    xenon2* & 7.95 & 1.34 \\
%    \hline
%    venkat01 & 4.79 & 2.02 \\
%    \hline
%    QC2534 & 0.61 & 0.24 \\
%    \hline
%    mplate* & 4.00 & 3.68 \\
%    \hline
%    light\_in\_tissue & 8.05 & 0.85 \\
%    \hline
%    kim1 & 2.68 & 0.59 \\
%    \hline
%    chevron2 & 17.01 & 2.52 \\   
%    \hline    
%\end{tabular}
%\caption{Ratio between total wall-clock time spent to solve least-squares problems for the Anderson mixing and the total wall-clock time spent to perform a Richardson's step when the ILU(0) preconditioner (left) and the ILUT($\tau$) preconditioner are used.}
%\label{ls_vs_richardson_cost}
%\end{table}

\begin{table}

\center
{
\begin{tabular}{|c||p{1.8cm}|p{1.8cm}||p{1.8cm}|p{1.8cm}||p{1.8cm}|p{1.8cm}|}
    \hline
         &  \multicolumn{6}{|c|}{\Large $\frac{\textbf{Total wall-clock time for least-squares}}{\textbf{Total wall-clock time for Richardson's steps}}$}\\
    \hline
    & \multicolumn{2}{|c|}{\bf Alternating AA{$(m=20)$}} & \multicolumn{2}{|c|}{\bf Subselect. Altern. AA{$(m=20)$}} & \multicolumn{2}{|c|}{\bf Random. Altern. AA{$(m=20)$}}\\
    \hline
        \textbf{Matrix} & \textbf{ILU(0)} & \textbf{ILUT($\tau$)} & \textbf{ILU(0)} & \textbf{ILUT($\tau$)} & \textbf{ILU(0)} & \textbf{ILUT($\tau$)} \\       
    \hline  
    fidap029 & 1.40 & 0.45 & 0.63 & 0.22 & 0.62 & 0.20\\
    \hline
    raefsky5 & 1.05 & 0.41 & 0.51 & 0.19 & 0.52 & 0.17\\
    \hline
    bcsstk29 & 0.34 & 0.35 & 0.25 & 0.12 & 0.23 & 0.11\\    
    \hline
    sherman3 & 2.19 & 0.41 & - & 0.21 & 0.37 & 0.22\\
    \hline
    sherman5 & 1.23 & 0.53 & 1.14 & 0.53 & 1.12 & 0.50\\
    \hline
    chipcool0 & 2.29 & 0.76 & 0.57 & 0.18 & 0.55 & 0.15\\
    \hline
    e20r0000 & 2.38 & 1.11 & 2.01 & 0.69 & 1.99 & 0.59\\
    \hline
    spmsrtls & 1.09 & 1.05 & 0.59 & 0.52 & 0.56 & 0.51\\
    \hline
    garon1 & 1.62 & 1.10 & 0.83 & 0.55 & 0.80 & 0.52\\
    \hline
    garon2 & 1.60 & 1.01 & - & 0.65 & 0.85 & 0.61\\
    \hline
    memplus & 5.04 & 0.92 & 2.28 & 0.47 & 2.25 & 0.43\\
    \hline
    saylr4 & 3.22 & 3.20 & 1.63 & 1.61 & 1.60 & 1.56\\     
    \hline
    xenon1 & - & - & - & - & - & - \\
    \hline
    xenon2 & - & - & - & - & - & - \\
    \hline
    venkat01 & 1.60 & 0.66 & 0.57 & 0.10 & 0.56 & 0.10\\
    \hline
    QC2534 & 0.20 & 0.07 & 0.13 & 0.05 & 0.11 & 0.04\\
    \hline
    mplate & - & - & - & - & - & -\\
    \hline
    light\_in\_tissue & 2.69 & 0.27 & 0.21 & 0.11 & 1.32 & 0.19\\
    \hline
    kim1 & 1.21 & 0.39 & 0.24 & 0.23 & 0.25 & 0.21\\
    \hline
    chevron2 & 5.52 & 1.21 & - & - & 2.17 & 0.43 \\   
    \hline    
\end{tabular}
}
\caption{{Ratio between total wall-clock time to solve least-squares problems and to perform Richardson's steps within a complete run of Alternating AA {$(m=20)$}, Subselected Alternating AA {$(m=20)$}, and Randomized Alternating AA {$(m=20)$} when ILU(0) and ILUT($\tau$) preconditioners are used. The dashed line indicates examples where the tested solver did not converge to the desired accuracy.}}
\label{ls_vs_richardson_cost_m20}
\end{table}

\subsection{Non-linear deterministic fixed-point iteration: Picard iteration for non-linear time-dependent Boltzmann equation}
\label{subsec:nonlinear}
\newcommand{\vect}[1]{\boldsymbol{#1}}
\def\bbR{\mathbb{R}}
\def\bbS{\mathbb{S}}
\def\SOL{\textsf{c}}
\newcommand{\EmAb}{\mbox{\tiny EA}}
\newcommand{\ES}{\mbox{\tiny ES}}
\newcommand{\IS}{\mbox{\tiny IS}}
\newcommand{\IN}{\mbox{\tiny IN}}
\newcommand{\OUT}{\mbox{\tiny OUT}}
\newcommand{\TOTAL}{\mbox{\tiny{Tot}}}       

{Even though the analysis in this paper focuses on the case of linear fixed-point problems, we test the variants of AA-based solvers considered in the last section on non-linear fixed-point problems and report the results in this section. 
Here}
we chose the non-linear systems that arise from implicit time discretization of the non-relativistic Boltzmann equation, which is a widely used model for neutrino transport in nuclear astrophysics applications \cite{MEZZACAPPA1999281}. 
The non-relativistic Boltzmann equation is given by
\begin{equation}\label{eq:Boltzmann}
  \partial_t{f} + \mathcal{T}(f) = \mathcal{C}(f)\:,
\end{equation}
where $f(\vect{x},\omega,\varepsilon,t)$ denotes the neutrino distribution function that describes the density of neutrino particles at position $\vect{x}\in\bbR^3$
%with momentum $\vect{p}\in\bbR^3$ 
traveling along direction $\omega\in\bbS^2$ with energy $\varepsilon\in\bbR^{+}$
at time $t\in\bbR^+$.
Here the advection and collision operators are denoted by $\mathcal{T}$  and $\mathcal{C}$, respectively. 
Implicit-explicit (IMEX) time integration schemes \cite{pareschiRusso_2005} are often used to solve \eqref{eq:Boltzmann}, in which the collision term is treated implicitly to relax the excessive explicit time-step restriction, while an explicit advection term is used to avoid spatially coupled non-linear systems.
With a simple backward Euler method, the implicit stage of an IMEX scheme applied to \eqref{eq:Boltzmann} at time $t^n$ takes the form
\begin{equation}\label{eq:Implicit}
    f(\vect{x},\omega,\varepsilon,t^{n+1}) = f(\vect{x},\omega,\varepsilon,t^{n}) + \Delta t \,\mathcal{C}(f(\vect{x},\cdot,\cdot,t^{n+1}))\:, \quad\forall\vect{x}\in\bbR^3\:,
\end{equation}
where we follow the existing approach \cite{Paul_Laiu_2021} and write the (space-time homogeneous) neutrino collision operator $\mathcal{C}$ as 
\begin{equation}
    \mathcal{C}(f) = {\eta_{\TOTAL}}(f) -  {\chi_{\TOTAL}}(f)\,f\:,
\end{equation}
where ${\eta_{\TOTAL}}$ denotes the total emissivity and ${\chi_{\TOTAL}}$ denotes the total opacity.
% \begin{equation}
%  \begin{alignedat}{2}
%   {\eta_{\TOTAL}}(f) :=& \hat{\eta}(\epsilon) +
%     {\epsilon^2}\int_{\bbS^2} R^{\IS}(\omega\cdot\omega',\epsilon) f(\omega',\epsilon)\,d\omega' +
%     \int_{\bbR^{+}}\int_{\bbS^{2}}R^{\IN}(\omega\cdot\omega', \epsilon, \epsilon') f (\omega',\epsilon')\,d\omega'\,{\epsilon'}^2 d{\epsilon'}\\
%     {\chi_{\TOTAL}}(f) := &\hat{\eta}(\epsilon) + \hat{\chi}(\epsilon) +
%     {\epsilon^2} \int_{\bbS^2} R^{\IS}(\omega\cdot\omega',\epsilon) \,d\omega' +
%     \int_{\bbR^{+}}\int_{\bbS^{2}}R^{\IN}(\omega\cdot\omega', \epsilon, \epsilon') f (\omega',\epsilon')\,d\omega'\,{\epsilon'}^2 d{\epsilon'} \\
%     &+ \int_{\bbR^{+}}\int_{\bbS^{2}}R^{\OUT}(\omega\cdot\omega', \epsilon, \epsilon')\,
%     \big(1-f (\omega',\epsilon')\big)\,d\omega'\,{\epsilon'}^2 d{\epsilon'}\:.
% \end{alignedat}
% \end{equation}
% Here $\hat{\eta}$ is the emissivity, $\hat{\chi}$ is the absorption opacity, $R^{\IS}$ is the elastic (isoenergetic) scattering kernel, and $R^{\IN}$ and $R^{\OUT}$ are the neutrino-electron scattering kernels.
We then formulate \eqref{eq:Implicit} at each time step into a non-linear fixed-point problem \cite{Paul_Laiu_2021}
\begin{equation}\label{eq:fixed_point_Boltzmann}
    f^{n+1} = G(f^{n+1})\, \text{ with }\,
    G(f^{n+1}):=({f^{n}+\Delta t \eta_{\TOTAL}(f^{n+1})})/({1+\Delta t \chi_{\TOTAL}(f^{n+1})})\:.
\end{equation}
This formulation ensures that the fixed-point operator $G$ is a contraction and allows for the standard Picard iteration as well as variants of AA to solve \eqref{eq:Implicit}.

In the numerical tests, we solve the spatially-homogeneous implicit system \eqref{eq:fixed_point_Boltzmann} at one time step, in which we apply a discrete ordinate angular discretization with a 110-point Lebedev quadrature rule on $\bbS^2$ and 64 geometrically progressing energy nodes in $[0,300]$ (MeV). 
Here we compare the performance of five iterative solvers: Picard iteration, standard AA, Alternating AA, and two versions of Reduced Alternating AA -- Subselected Alternating AA and Randomized Alternating AA. 
The history length $m=3$ is used for all four AA variants, and the alternating ones perform Anderson mixing every $p=3$ iterations. The Subselected Alternating AA and Randomized Alternating AA solvers are implemented as presented in Section~\ref{section:reduced_aar} with the reduced dimension $s$ chosen to satisfy the bound \eqref{error_bound2}, where $\epsilon=10^{-8}$ is used.
Each solver is tested on problems at six different matter density values. In general, higher matter densities correspond to stronger collision effects, which make the implicit system more stiff and require more iterations to converge. 
In the test, the total emissivity and opacity are computed using opacity kernels from literature
\cite{bruenn_1985}.

Figure~\ref{fig:Boltzmann} reports the iteration count and computation time required for each compared solver at various matter densities. These reported results are averaged over 30 runs. Figure~\ref{fig:Boltzmann_Iter} confirms that AA and Alternating AA require significantly fewer iterations to reach the prescribed $10^{-10}$ relative tolerance than the Picard iteration, especially for problems at higher densities (more stiff). Randomized Alternating AA essentially resembles the iteration counts of Alternating AA, while the Subselected Alternating AA requires much higher iteration counts than other AA variants.
Figure~\ref{fig:Boltzmann_Time} illustrates that AA actually takes more computation time than Picard iteration, which indicates that the additional cost of solving least-squares problems outweighs the reduction of iteration counts. Alternating AA lowers the computation time of AA by reducing the frequency of solving the least-squares problems, and Randomized Alternating AA further improves the computation time by approximating the least-squares solution with solution to a sparsely projected smaller problem.
{The results reported in Table~\ref{tab:LS_eval_time_ratio_Boltzmann} confirm the conjecture above by showing that, on average, the computation time for solving the least-squares problems in AA ($t_{\text{LS}}$) is indeed higher than the computation time for evaluating the fixed-point operator ($t_{G}$) in this example. From Table~\ref{tab:LS_eval_time_ratio_Boltzmann}, the ratio, ${t_{\text{LS}}}/{t_{G}}$, is lowered in Alternating AA and is further reduced in the both the Subselected and Randomized Alternating AA solvers, which is reflected in the total computation time in Figure~\ref{fig:Boltzmann_Time} for solvers share similar iteration counts.}

\begin{figure}[h]
  \centering
  \subfloat[]{\label{fig:Boltzmann_Iter}\includegraphics[width=0.475\textwidth]{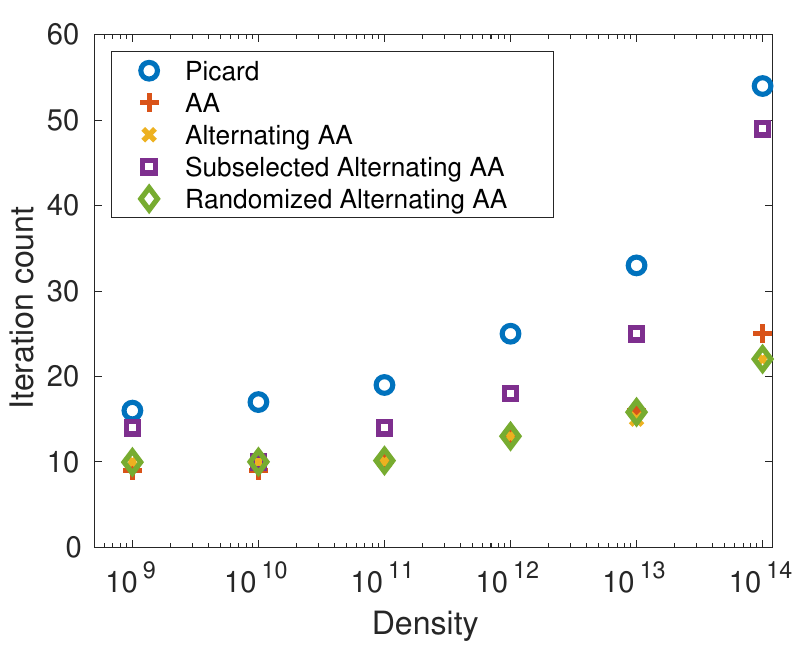}}~~
  \subfloat[]{\label{fig:Boltzmann_Time}\includegraphics[width=0.5\textwidth]{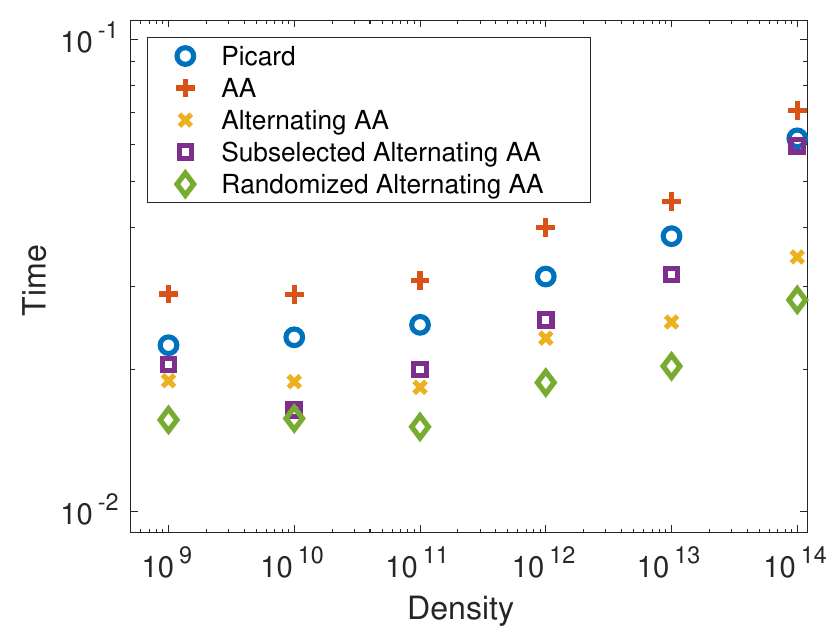}}  
  \caption{\label{fig:Boltzmann}Iteration counts and computation time (in $\log$ scale) for various fixed-point solvers for solving the non-linear system \eqref{eq:fixed_point_Boltzmann} at different matter density, which corresponds to the stiffness of the system.}
\end{figure}
\begin{table}[h]
{    
    \centering
    \begin{tabular}{|c||c|c|c|c|c|}
    \hline
         Solver & Picard & AA & Alternating AA & Subselected Alternating AA & Randomized Alternating AA \\
         \hline
         ${t_{\text{LS}}} / {t_{G}}$ &  0  &  2.04  &  0.80  &  0.23  &  0.25\\
         \hline
    \end{tabular}
    \caption{{ The ratio between $t_{\text{LS}}$ and $t_{G}$ for various fixed-point solvers. Here $t_{\text{LS}}$ and $t_{G}$ denote the computation time spent in solving the least-squares problems and in evaluating the fixed-point operator $G$, respectively. The reported results are averaged over 30 runs and over the six density values.}
    }
    \label{tab:LS_eval_time_ratio_Boltzmann}
}
\end{table}

\section{Conclusion and future development}
\label{section:conclusions}
Although AA has been shown to significantly improve the convergence of fixed-point iterations in several scientific applications, effectively performing AA for large scale problems without introducing excessive computational burden remains still a challenge. One possibility to reduce the computational cost consists in projecting the least-squares to compute the Anderson mixing onto a projection subspace, but this may compromise the convergence of the scheme. 

We derived rigorous theoretical bounds for AA that allow for efficient approximate calculations of the residual, which reduce computational time while maintaining convergence. The theoretical bounds provide useful insights to judiciously inject inaccuracy in the calculations whenever this allows to alleviate the computational burden without compromising the final accuracy of the physics based solver. In particular, we consider situations where the inaccuracy arises from projecting the least-squares problem onto a subspace. 

Guided by the theoretical bounds, we constructed a heuristic that  dynamically adjusts the dimension of the projection subspace at each iteration. As the process approaches convergence, the backward error decreases, which allows for an increase of the inaccuracy of the least-squares calculations while still maintaining the residual norm within the prescribed bound. This reduces the need for very accurate calculations needed to converge. 
While previous theoretical results provide error estimates when the accuracy is fixed a priori throughout the fixed-point scheme, thus potentially compromising the final attainable accuracy, our study allows for adaptive adjustments of the accuracy to perform AA while still ensuring that the final residual norm drops below a prescribed threshold. 
The resulting numerical scheme, called Reduced Alternating AA, is clearly preferable over standard Alternating AA when the computational budget is limited.

Numerical results on linear systems and non-linear non-relativistic Boltzmann equation show that our heuristic can be used to effectively reduce the computational time of Reduced Alternating AA without affecting the final attainable accuracy compared to traditional methods to apply AA. 

Future work will be dedicated to apply Reduced Alternating AA to non-linear stochastic fixed point iterations arising from: (i) Monte Carlo methods for solving the Boltzmann equation and (ii) training of physics informed graph convolutional neural networks for the prediction of materials properties from atomic information. 

\section*{Conflict of interest statement}
The authors do not have any conflict of interest to declare. 

\section*{Acknowledgement}
Massimiliano Lupo Pasini thanks Dr.\ Vladimir Protopopescu for his valuable feedback in the preparation of this manuscript.
Paul Laiu thanks Dr.\ Victor DeCaria for insightful discussions. 
This work was supported in part by the Office of Science of the Department of Energy, by the Exascale Computing Project (17-SC-20-SC), a collaborative effort of the U.S. Department of Energy Office of Science and the National Nuclear Security Administration, and by the Artificial Intelligence Initiative as part of the Laboratory Directed Research and Development (LDRD) Program of Oak Ridge National Laboratory managed by UT-Battelle, LLC for the US Department of Energy under contract DE-AC05-00OR22725.

%\nocite{*}% Show all bib entries - both cited and uncited; comment this line to view only cited bib entries;
\bibliography{wileyNJD-VANCOUVER}%

\section*{Author Biography}

\begin{biography}{\includegraphics[width=66pt,height=86pt]{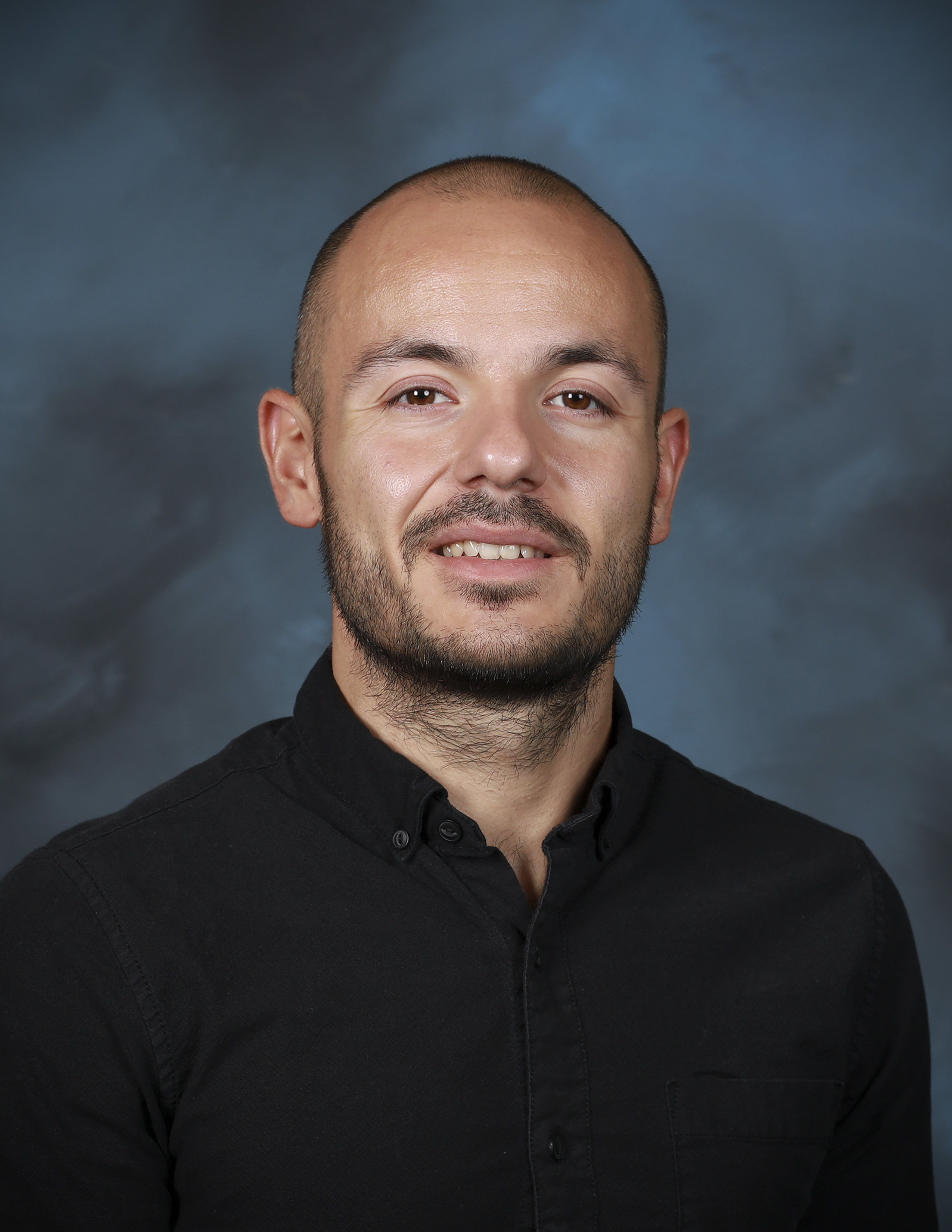}}{\textbf{Massimiliano Lupo Pasini.} Massimiliano (Max) Lupo Pasini obtained his Bachelor of Science and Master of Science in Mathematical Engineering at the Politecnico di Milano in Milan, Italy. The focus of his undergraduate and master studies was statistics and discretization techniques and reduction order models for partial differential equations. He obtained his PhD in Applied Mathematics at Emory University in Atlanta (GA) in May 2018. The main topic of his doctorate work was the development of efficient and resilient linear solvers for upcoming computing architectures moving towards exascale. Upon graduation, Max joined the Oak Ridge National Laboratory (ORNL) as a Postdoctoral Researcher Associate in the Scientific Computing Group at the National Center for Computational Sciences (NCCS). Since 2020 Max has been a Data Scientist in the Scalable Algorithms and Coupled Physics Group in the Advanced Computing Methods for Engineered Systems Section of the Computational Sciences and Engineering Division at ORNL. Max’s research focuses on the development of surrogate models for material sciences, scalable hyper parameter optimization techniques for deep learning models, and acceleration of computational methods for physics applications. He is currently the lead of the Artificial Intelligence for Scientific Discovery thrust of the ORNL Artificial Intelligence Initiative.}
\end{biography}

\begin{biography}{\includegraphics[width=66pt,height=86pt]{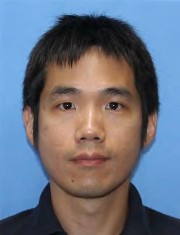}}{\textbf{M.~Paul Laiu.} Paul Laiu is a Staff Mathematician in the Multiscale Methods and Dynamics Group at Oak Ridge National Laboratory.  His research interest includes numerical optimization, surrogate modeling, and numerical schemes for various partial differential equations in kinetic theory. His work focuses on the development of mathematical tools that accelerate simulation of multiscale systems while preserving the structure of the multiscale phenomena. Paul received his Ph.D. degree in Electrical and Computer Engineering from University of Maryland College Park in 2016.  Before joining ORNL in 2017, he was a postdoctoral research associate at the Mathematics Department in University of Tennessee Knoxville.}
\end{biography}

\end{document}